\documentclass{amsart}

\usepackage[margin=1in]{geometry}
\usepackage[utf8]{inputenc}
\usepackage{amsmath, amssymb, amsthm}
\usepackage{graphicx}
\usepackage{dsfont}
\usepackage{xcolor}
\usepackage{dirtytalk}
\usepackage{comment}
\usepackage{orcidlink}
\usepackage{hyperref}
\usepackage{mathtools}
\usepackage[shortlabels]{enumitem}

\numberwithin{equation}{section}

\newcommand{\supp}{\operatorname{supp}}

\newcommand{\Aaa}{{\mathcal{A}}}

\newcommand{\N}{\mathbb{N}}

 \newcommand{\cM}{\mathcal M}

\newcommand{\dist}{\operatorname{dist}}

\newcommand{\tr}{\operatorname{Tr}}

\providecommand{\dist}{\operatorname{dist}}

\newtheorem{assump}{}

\let\TeXchi\chi
\newbox\chibox
\setbox0 \hbox{\mathsurround0pt $\TeXchi$}
\setbox\chibox \hbox{\raise\dp0 \box 0 }
\def\chi{\copy\chibox}

\newtheoremstyle{break}
  {\topsep}{\topsep}%
  {\itshape}{}%
  {\bfseries}{}%
  {\newline}{}%
\theoremstyle{break}

\newcommand{\Assump}[1]{\textup{(Assumption~\ref{#1})}}

\theoremstyle{definition}

\newcommand{\R}{\mathbb{R}}

\usepackage{tikz}



\usepackage{tikz}

\theoremstyle{plain}
\newtheorem{theorem}{Theorem}[section]
\newtheorem{lemma}[theorem]{Lemma}
\newtheorem{proposition}[theorem]{Proposition}

\newtheorem{corollary}[theorem]{Corollary}

\theoremstyle{definition}
\newtheorem{definition}[theorem]{Definition}

\theoremstyle{remark}
\newtheorem{remark}[theorem]{Remark}

\numberwithin{equation}{section}

\title[Quantitative Hopf-Oleinik lemma for degenerate fully nonlinear operators]{A quantitative Hopf-Oleinik lemma for degenerate fully nonlinear operators and applications to free boundary problems}
\author[Davide Giovagnoli, Enzo Maria Merlino, Diego Moreira]{Davide Giovagnoli$^{\orcidlink{0009-0005-3375-5199}}$, Enzo Maria Merlino$^{\orcidlink{0000-0001-8501-9613}}$, \and Diego Moreira}

\newcommand{\Addresses}{{
  \bigskip
  \footnotesize

    Davide Giovagnoli {\orcidlink{0009-0005-3375-5199}}, \textsc{Dipartimento di Matematica, Universit\`a di Bologna, Piazza di Porta S.Donato 5, 40126, Bologna, BO, Italy}\par\nopagebreak
  \textit{E-mail address:} {\tt d.giovagnoli@unibo.it}\\

  \medskip

    Enzo Maria Merlino \orcidlink{0000-0001-8501-9613}, \textsc{Dipartimento di Matematica, Universit\`a di Bologna, Piazza di Porta S.Donato 5, 40126, Bologna, BO, Italy}\par\nopagebreak
  \textit{E-mail address:} {\tt enzomaria.merlino2@unibo.it}\\
\medskip

  Diego Moreira, \textsc{Departamento de Matematica,
Universidade Federal do Ceará, Av. Humberto Monte Pici
60455-760, Fortaleza, CE, Brasil}\par\nopagebreak
  \textit{E-mail address:} {\tt dmoreira@mat.ufc.br}\\

 }

}


\date{\today}
\dedicatory{Dedicated to Sandro Salsa on his $75^{th}$ birthday}

\begin{document}

\begin{abstract}
We prove a quantitative inhomogeneous Hopf-Oleinik lemma for viscosity solutions of 
\[
|\nabla u|^{\alpha}F(D^{2}u)=f
\] and, more generally, for viscosity supersolutions of
$|\nabla u|^{\alpha}\,\mathcal M^-_{\lambda,\Lambda}(D^{2}u)\le f$.
The result yields linear boundary growth with universal constants depending only on the structural data.
We also exhibit a counterexample showing that the Hopf lemma fails for equations that act only in the large–gradient regime (in the sense of Imbert and Silvestre), thereby delineating the scope of our theorem.
As applications, we obtain Lipschitz regularity for viscosity solutions of one–phase Bernoulli free boundary problems driven by these degenerate fully nonlinear operators and derive $\varepsilon$–uniform Lipschitz bounds for a one–phase flame propagation model.
\end{abstract}
\maketitle

\section{Introduction}

The Hopf-Oleinik lemma is a cornerstone of the theory of elliptic partial differential equations. In its classical form, independently proved by Hopf in \cite{hopf1952} and Oleinik in \cite{Oleinik1952}, it asserts that a nonnegative supersolution of a second order uniformly elliptic equation that vanishes at a boundary point must exhibit a strictly positive inward normal derivative at that point. This qualitative principle underlies the strong maximum principle and has wide impact on boundary regularity and free boundary analysis. Throughout this paper, we use the \emph{ball as the prototype of a $C^{1,1}$ domain}, which captures the local geometry near a smooth boundary point and suffices for the local argument.

Our main goal is to prove a \emph{quantitative} Hopf-Oleinik lemma for 
viscosity solutions of 
\begin{equation}\label{eq:mainB_1}
|\nabla u|^{\alpha}\,F(D^2u)=f\quad \text{in }B_1,
\end{equation}
where $F$ is a second-order fully nonlinear elliptic operator and $\alpha\geq 0$.
More generally, this result applies to viscosity supersolutions of
\begin{equation} \label{eq:main_supersol}
|\nabla u|^{\alpha}\,\mathcal M^-_{\lambda,\Lambda}(D^{2}u)\le f \quad \text{in }B_1,
\end{equation}
where \(\mathcal M^-_{\lambda,\Lambda}\) denotes the Pucci minimal operator with ellipticity constants $0<\lambda \leq \Lambda$.
In this setting, whenever \(u\ge0\) in \(B_1\) and \(u(x_0)=0\) for some \(x_0\in\partial B_1\), the Hopf principle holds \emph{quantitatively}: namely there exist \(\varepsilon>0\) and \(C>0\), depending only on $n,\lambda,\Lambda$ and $\alpha$, such that
\[
\|u\|_{L^{\varepsilon}(B_{1/2})}
\;\le\;
C\bigl(
\partial_\nu u(x_0)
+\|f^{+}\|_{L^{\infty}(B_1)}^{\frac{1}{1+\alpha}}
\bigr),
\]
which recovers the classical Hopf-Oleinik lemma when \(f\le0\).
In addition, we obtain a non-infinitesimal quantitative inhomogeneous Hopf-Oleinik inequality. More precisely, there exist constants \(A_{1},A_{2}>0\) and \(\varepsilon>0\), depending only on \(n,\lambda,\Lambda,\alpha\), such that for all \(x\in B_{1}\)
\[
u(x)\ge
\left(
A_{1}\,\|u\|_{L^{\varepsilon}(B_{1/2})}
-
A_{2}\,\|f^{+}\|_{L^{\infty}(B_{1})}^{\frac{1}{1+\alpha}}
\right)
\operatorname{dist}(x,\partial B_{1}).
\]

Furthermore, we will show how these results can be used to face regularity issues in free boundary problems, including Lipschitz regularity of the one-phase free boundary problem driven by degenerate elliptic operators and $\varepsilon$-uniform estimates for a singular perturbation model arising in combustion theory.
We also analyze the failure of the Hopf-Oleinik lemma for equations that are enforced only in the large-gradient regime, in the sense of Imbert and Silvestre in \cite{Imbert-Silvestre2016}.
These applications will be presented in more detail in the next section.

The \emph{quantitative} Hopf-Oleinik lemma for harmonic functions appears pervasively within the variational and viscosity theory of Bernoulli–type free boundary problems: in the variational approach of Alt and Caffarelli \cite{AC} and Alt, Caffarelli and Friedman \cite{ACF}, and in the viscosity framework developed by Luis Caffarelli in his foundational papers \cite{C1,C2,C3}. In view of this dual presence, it naturally appears, and is treated systematically, in the monograph of Caffarelli and Salsa \cite{CSbook}, where it is used as a key tool in the general theory of free boundary problems. Moreover, the scope of application of the Hopf lemma is considerably broader, being closely connected to boundary Harnack–type inequalities and to boundary Krylov-type regularity theorems for uniformly elliptic second order equations. A quantitative version of this Hopf-Oleinik principle was established by Berestycki, Caffarelli and Nirenberg in \cite[Theorem 2.2]{BCN}, where it provides $\varepsilon$–uniform Lipschitz estimates for regularized free–boundary problems of flame–propagation type in combustion theory. We refer to the works of Sirakov \cite{S1,S2,S3} on boundary Harnack–type inequalities, which encode the quantitative version of the Hopf–Oleinik lemma. In the fully nonlinear setting, these results are indeed more general. The literature contains several contributions concerning extensions of the quantitative Hopf–Oleinik lemma to broader classes of operators; see, for instance, \cite{BragaMoreiraAdv,MoreiraSantosSoares2024Fenici,Lian2024} for certain elliptic equations, and \cite{caffarelli2013some,torres2024parabolic,ferrari2025geometry,jesus2025fully} for some results in the parabolic setting. In this sense, our research is related to some of the ideas developed for the $g$-Laplacian in \cite{BragaMoreiraAdv}.

In this work, we focus on a class of fully nonlinear operators whose ellipticity degenerates along the sets of critical points. The equation \eqref{eq:mainB_1} belongs to a larger class of nonlinear elliptic equations considered in a series of work of Birindelli and Demengel, starting with \cite{Birindelli-Demengel2004} and then followed by \cite{Birindelli-Demengel2007,Birindelli-Demengel2009,Birindelli-Demengel2010,birindelliDemengel2010uniqueness,Birindelli2010eigenfunctions}. 
This prolific line of research has focused on topics such as the comparison principle, Liouville-type theorems, and issues related to existence and regularity of solutions for this class of equations, including the singular case, that is, for $-1<\alpha<0$. Since it is not possible to provide a complete review of all the works in this area, we only mention a few illustrative examples. For this we refer to \cite{davila2009alexandroff,davila2010harnack,junges2010aleksandrov,Imbert-Silvestre2013,araujo2015geometric}.
For this class of operators, a Hopf-type lemma has been established in \cite[Lemma 3.9]{birindelli2013overdetermined}. 
Our results extend, in a more precise and quantitative way, those in \cite{birindelli2013overdetermined}, while also covering the inhomogeneous case without any sign restriction on the right-hand side.
This quantitative refinement is fundamental for dealing with regularity and estimates in free boundary problems. 

This paper, aiming to be entirely self-contained, provides full proofs of all the results. In particular, since we were not able to find a reference that systematically addresses the weak Harnack inequality and the Harnack inequality for these operators in the case $\alpha\geq 0$, we include this result for completeness in Section \ref{s:5}. The Harnack inequality in the degenerate case in dimension two was obtained by Birindelli and Demengel in \cite{Birindelli2010eigenfunctions}, while in the singular case it was established by Dávila, Felmer, and Quaas in \cite{davila2010harnack}.
It is probably well known by experts in the subject that these inequalities can be obtained from the results of Imbert and Silvestre for equations that hold only where the gradient is large, but for the sake of completeness, we provide a proof.

Following the same principle, in Section \ref{s:glue-SVA} we provide a gluing result for Sobolev functions across rough interfaces, which in particular yields the Sobolev regularity of the extension of the positive phase \(u^{+}\). Although this type of result is commonly known in the literature for functions whose traces coincide along Lipschitz hypersurfaces, see for instance \cite[Lemma A.8]{KinderlehrerStampacchia} and \cite[Theorem 5.8]{EvansGariepy} for $BV$ functions or \cite[Theorem 3.84]{AmbrosioFuscoPallara} for finite perimeter interface, we could not find it in the specific formulation needed for the study of the one-phase Bernoulli problem without additional assumptions on the free boundary. Therefore, we decided to include it as an auxiliary result, which might be of independent interest.

Finally, in Appendix  \ref{appendixA}, we reconcile the notions of classical \(C^{1,\gamma}\) regularity, \(C^{1,\gamma}\) via polynomial approximation, and Campanato \(C^{1,\gamma}\), providing sharper, fully quantified estimates.
A version of this comparison appears in the appendix of \cite{BFM}, and here we complete and refine the bounds.

\textbf{Organization of the paper.}
Section~\ref{s:2} states our quantitative Hopf-Oleinik results in the degenerate fully nonlinear setting, including the formulation for viscosity supersolutions of $|\nabla u|^{\alpha}\,\mathcal M^-_{\lambda,\Lambda}(D^{2}u)\le f$ and their normal--derivative/sup variants. Building on these results, this section presents the applications to free boundary problems. Section~\ref{sec:notation} fixes notation and the framework. We also discuss the notion of viscosity solution, introduced by Birindelli and Demengel and its relation to the viscosity notion in the large gradient regime of Imbert and Silvestre. 
Section~\ref{s:4} builds an annular Pucci barrier with quantitative geometric control (height and slope). Section~\ref{s:5} establishes a weak $L^{\varepsilon}$ estimate and a Harnack inequality for the equation \eqref{eq:mainB_1}.  Section \ref{s:6} gives the proof of the main quantitative Hopf-Oleinik lemma. 
Section~\ref{s:7} presents a counterexample showing that Hopf fails for equations acting only on the large-gradient set.
Section~\ref{s:8} derives interior gradient bounds for \eqref{eq:mainB_1}. Section~\ref{s:glue-SVA} proves a Sobolev gluing lemma across rough interfaces.
Section~\ref{s:10} applies the main result to obtain Lipschitz regularity in a one phase free boundary problem. Section~\ref{s:11} treats a flame-propagation model and obtains $\varepsilon$-uniform Lipschitz bounds. Appendix \ref{appendixA} provides a Campanato-type criterion implying classical $C^{1,\omega}$ regularity.

\section{Main Results}\label{s:2}

This section collects the main results of the paper. We begin with the quantitative, inhomogeneous Hopf–Oleinik lemma for the degenerate fully nonlinear equation \eqref{eq:mainB_1} and its weakened version for viscosity supersolutions of \eqref{eq:main_supersol}.

\begin{theorem}[Quantitative Hopf-Oleinik lemma]\label{HOL-quantitative}
	Let $\alpha\ge0$ and $0<\lambda\le\Lambda$.
	Let $u\in C(\overline{B_1})$ be a viscosity supersolution of
	\begin{equation}\label{eq-HOL}
	|\nabla u|^{\alpha}\,\mathcal M^-_{\lambda,\Lambda}(D^2u)\le f
	\qquad\text{in }B_1,
	\end{equation}
	where $f\in C(B_1)\cap L^\infty(B_1)$ and $u\ge0$ in $B_1$.
	Then there exist constants $\varepsilon\in(0,1)$ and $A_1,A_2,C>0$,
	depending only on $n,\lambda,\Lambda,\alpha$, such that:
	\begin{enumerate}
		\item[\textnormal{(A)}] \textbf{Linear growth with respect to the distance from the boundary.}
		For every $x\in B_1$,
		\begin{equation}\label{eq:HOL-1}
		u(x)\;\ge\;
		\left(
		A_1\,\|u\|_{L^{\varepsilon}(B_{1/2})}
		-A_2\,\|f^{+}\|_{L^{\infty}(B_1)}^{\frac{1}{1+\alpha}}
		\right)
		\operatorname{dist}\bigl(x,\partial B_1\bigr).
		\end{equation}
		
		\item[\textnormal{(B)}] \textbf{Control by the interior normal derivative.}
		If for some $x_0\in\partial B_1$ with $u(x_0)=0$ the interior normal derivative
		$\partial_\nu u(x_0)$ exists, then
		\begin{equation}\label{eq:HOL-2}
		\|u\|_{L^{\varepsilon}(B_{1/2})}
		\;\le\;
		C\left(
		\partial_\nu u(x_0)
		+\|f^{+}\|_{L^{\infty}(B_1)}^{\frac{1}{1+\alpha}}
		\right).
		\end{equation}
		
		\item[\textnormal{(C)}] \textbf{Supremum version for a double differential inequality.}
		If $u$ satisfies in $B_1$ the two–sided Pucci inequality
		\begin{equation}\label{eq:HOL-3}
		\,|\nabla u|^{\alpha}\,\mathcal M^-_{\lambda,\Lambda}(D^2u)
		\;\le\; f
		\;\le\;
		|\nabla u|^{\alpha}\,\mathcal M^+_{\lambda,\Lambda}(D^2u),
		\end{equation}
		then the estimates of \textnormal{(A)} and \textnormal{(B)} hold with
		$\|u\|_{L^{\varepsilon}(B_{1/2})}$ replaced by $\displaystyle\sup_{B_{1/2}} u$.
		
		\item[\textnormal{(D)}] \textbf{Unique continuation.}
		If $f\le0$ in $B_1$ and, for some $x_0\in\partial B_1$,
		the interior normal derivative $\partial_\nu u(x_0)$ exists and $\partial_\nu u(x_0)=0$, then
		\[
		u\equiv0
		\qquad\text{in }B_1.
		\]
	\end{enumerate}
\end{theorem}

The following result shows that the Hopf lemma does not hold for equations acting only on the large-gradient set, in the sense of Imbert and Silvestre, see \cite{Imbert-Silvestre2016}.
 
\begin{proposition}[Failure of the Hopf lemma in the large gradient regime]\label{prop:hopf-fails-large-gradient-only}
	Fix $n\ge2$ and set $u(x)=(1-|x|)^2$. Then
	\[
	u\in C(\overline{B_1})\cap C^\infty(B_1\setminus\{0\}),\qquad
	u\ge0\ \text{in }B_1,\quad u=0\ \text{on }\partial B_1,
	\]
	and, in the viscosity sense,
	\[
	\Delta u\ \le\ 0 \quad \text{in }\ \{x\in B_1:\ |\nabla u(x)|\ge 1\}.
	\]
	Moreover, for every $\xi\in\partial B_1$, if $\nu$ denotes the inner unit normal at $\xi$, then $\partial_\nu u(\xi)=0$.
	Hence, Hopf’s boundary point conclusion fails in this large gradient regime.
\end{proposition}
As a consequence of Theorem \ref{HOL-quantitative}, which provides a quantitative boundary nondegeneracy, we derive Lipschitz regularity for viscosity solutions to one-phase Bernoulli problem governed by the equation \( |\nabla u|^{\alpha}F(D^{2}u)=f \).  This classical free boundary problem originates from the seminal work of Alt and Caffarelli \cite{AC} in the variational setting. For an extensive presentation on the Bernoulli one-phase problem, we refer to \cite{velichkov2023regularity}.  Several extensions have since been established, including non-variational frameworks such as the one considered here. The result below extends, under the assumption of a bounded right-hand side, the fully nonlinear case $\alpha=0$, previously addressed in \cite{MoreiraWangArma2014} and \cite{BragaMoreiraCalcVar2022}.
\begin{theorem}[Lipschitz regularity for a one-phase Bernoulli problem]\label{thm:Lip-FBP}
	Let $\alpha\geq 0$ and let $u\in C(B_{1})$ be a viscosity solution of
	\begin{equation} \label{eq:FBP1}
	\begin{cases}
		|\nabla u|^{\alpha}\,F(D^{2}u)=f & \text{in } B_{1}^{+}(u):=\{u>0\}\cap B_{1},\\[3pt]
		|\nabla u^{+}|\le h & \text{on } \mathcal{F}(u):=\partial\{u>0\}\cap B_{1},
	\end{cases}
	\end{equation}
	where $F$ is uniformly elliptic with constants $(\lambda,\Lambda)$ and $F(0)=0$,
	$f\in C(B_{1})\cap L^{\infty}(B_{1})$, and $h\in C(\mathcal{F}(u))\cap L^{\infty}(\mathcal{F}(u))$.
	Then $u^{+}\in C^{0,1}(B_{1/2})$ and there exists a constant
	$C=C(n,\lambda,\Lambda,\alpha)>0$ such that
	\[
	\|\nabla u^{+}\|_{L^{\infty}(B_{1/2})}
	\;\le\;
	C\Big(
	\|h\|_{L^{\infty}(\mathcal{F}(u))}
	+\|u\|_{L^{\infty}(B_{1}^{+}(u))}
	+\|f\|_{L^{\infty}(B_{1})}^{\frac{1}{1+\alpha}}
	\Big).
	\]
	Moreover, if $0\in \mathcal{F}(u)$, the estimate holds without the term
	$\|u\|_{L^{\infty}(B_{1}^{+}(u))}$.
\end{theorem}
Finally, we apply these ideas to a one-phase flame-propagation model arising in combustion theory, deriving \(\varepsilon\)-uniform Lipschitz bounds. Such a model naturally arises in the description of laminar flames as an asymptotic limit for large activation energy with source terms, corresponding to the singular limit case as $\varepsilon \to 0.$
Related results were previously obtained by Araújo, Ricarte, and Teixeira in \cite{ART2017AnnalesPoinc} for Perron-type solutions of a flame-propagation model associated with the operator under consideration. In contrast, we show that Lipschitz regularity holds for every viscosity solution of the one-phase problem. In the case $\alpha=0$, a more general result was established in \cite{MoreiraWangCalcVar}. 

\begin{theorem}[Lipschitz regularity for the flame propagation problem]\label{thm:flame-deg}
	Let $\alpha\geq 0$ and let $u_\varepsilon\in C(B_{1})$ be a bounded nonnegative viscosity solution of
	\begin{equation} \label{eq:flame_prop}
	|\nabla u_\varepsilon|^{\alpha}\,F(D^{2}u_\varepsilon)
	=\beta_\varepsilon(u_\varepsilon)+f
	\quad\text{in } B_{1},
	\end{equation}
	where $F$ is uniformly elliptic with constants $\lambda,\Lambda$, $F(0)=0$,
	$f\in C(B_{1})\cap L^{\infty}(B_{1})$,
	and $\beta_\varepsilon(t):=\varepsilon^{-1}\beta(t/\varepsilon)$ with
	$\beta\in C(\R)$ bounded and supported in $[0,1]$. Then $u_\varepsilon\in C^{1,\beta_0}(B_{1})$ locally for some
	$\beta_0=\beta_0(n,\lambda,\Lambda,\alpha)\in(0,1)$.
	Moreover, for every $\varepsilon\in(0,\tfrac18]$ there exists a constant
	$C=C(n,\lambda,\Lambda,\alpha)>0$, independent of $\varepsilon$ and of $\beta$, such that
	\begin{equation} \label{eq:flameprop1}
	\|\nabla u_\varepsilon\|_{L^{\infty}(B_{1/2})}
	\;\le\;
	C\Big(
	1
	+\|\beta\|_{L^{\infty}(\R)}^{\frac{1}{1+\alpha}}
	+\|u_\varepsilon\|_{L^{\infty}(B_{1})}
	+\|f\|_{L^{\infty}(B_{1})}^{\frac{1}{1+\alpha}}
	\Big).
	\end{equation}
	In particular, if the sequence $(u_\varepsilon)_{\varepsilon}$ is uniformly bounded in $L^{\infty}(B_{1})$, then it is Lipschitz in $B_{1/2}$ uniformly in $\varepsilon$.
		\noindent Moreover, if $\varepsilon\in(0,\tfrac18]$ and $u_\varepsilon(0)\le\varepsilon$,
then it holds that
	\[
	\|\nabla u_\varepsilon\|_{L^{\infty}(B_{1/4})}
	\;\le\;
	C\Big(
	1
	+\|\beta\|_{L^{\infty}(\R)}^{\frac{1}{1+\alpha}}
	+\|f\|_{L^{\infty}(B_{1})}^{\frac{1}{1+\alpha}}
	\Big),
	\]
	that is, the term $\|u_\varepsilon\|_{L^{\infty}(B_{1})}$ can be dropped from the estimate \eqref{eq:flameprop1}.
\end{theorem}

\section{Notation and preliminaries}\label{sec:notation}

Let the space dimension be denoted by \(n\ge 2\). We write \(\mathcal S^n\) for the set of real symmetric \(n\times n\) matrices, endowed with the operator norm \(|\cdot|\).
For \(X\in\mathcal S^n\), let \(\{\mu_i\}_{i=1,\dots,n}\) be its eigenvalues and decompose
\[
X^{+}:=\sum_{\mu_i>0}\mu_i\,e_i\otimes e_i,
\qquad
X^{-}:=-\sum_{\mu_i<0}\mu_i\,e_i\otimes e_i,
\qquad
X=X^{+}-X^{-},
\]
so that \(X^{\pm}\in\mathcal S^n\) are nonnegative and \(\tr(X^{\pm})=\sum_{\mu_i\gtrless 0}|\mu_i|\).
We use the partial order \(X\ge Y\) to mean \(X-Y\) is nonnegative definite.
By $x$ we denote vectors in \(\mathbb R^n\) and \(|x|\) their Euclidean norm; gradients and Hessians are
\(\nabla u\) and \(D^2u\), respectively. We denote with \(B_r(x_0)\) the ball centered at $x_0 \in \R^n$ and radius $r\geq 0$ and we define \(B_r:=B_r(0)\).

\textbf{Pucci extremal operators.}
Fix ellipticity constants \(0<\lambda\le \Lambda\).
For \(X\in\mathcal S^n\), the Pucci extremal operators are
\begin{equation}\label{eq:Pucci}
\mathcal M^+_{\lambda,\Lambda}(X):=\Lambda\,\tr(X^{+})-\lambda\,\tr(X^{-}),
\qquad
\mathcal M^-_{\lambda,\Lambda}(X):=\lambda\,\tr(X^{+})-\Lambda\,\tr(X^{-}).
\end{equation}

Throughout the work, we will always consider $(\lambda,\Lambda)-$elliptic operators \(F:\mathcal S^n\to\mathbb R\) 
according to the following definition of uniform ellipticity. Without loss of generality, we will always assume that $F(0)=0$. 

\begin{assump}[Uniform ellipticity]\label{assump:Pucci-ellipticity}
A continuous operator \(F:\mathcal S^n\to\mathbb R\) is \((\lambda,\Lambda)\)-elliptic if for all \(X,Y\in\mathcal S^n\),
\begin{equation}\label{eq:Pucci-ellipticity}
\mathcal M^-_{\lambda,\Lambda}(X-Y)\ \le\ F(X)-F(Y)\ \le\ \mathcal M^+_{\lambda,\Lambda}(X-Y).
\end{equation}
\end{assump}

\begin{remark}\label{rem:older-ellipticity}
Uniform ellipticity is often expressed in an alternative, yet equivalent, form. 
Indeed, taking \(Y=X+N\) with \(N\ge 0\) in \eqref{eq:Pucci-ellipticity} yields
\[
\mathcal M^-_{\lambda,\Lambda}(N)\ \le\ F(X+N)-F(X)\ \le\ \mathcal M^+_{\lambda,\Lambda}(N).
\]
Since \(N\ge 0\) implies \(\mathcal M^+_{\lambda,\Lambda}(N)=\Lambda\,\tr(N)\) and
\(\mathcal M^-_{\lambda,\Lambda}(N)=\lambda\,\tr(N)\), we recover the estimate
\[
\lambda\,\tr(N)\ \le\ F(X+N)-F(X)\ \le\ \Lambda\,\tr(N),
\]
which matches the condition used for instance in \cite[Definition 2.1]{Caffarelli-Cabre1995},  up to identifying \(|N|\) with \(\tr(N)\) for \(N\ge 0\).
\end{remark}

\textbf{Degenerate fully nonlinear operators.}
Given $\alpha \geq 0$ and \(F\) as in Assumption~\ref{assump:Pucci-ellipticity}, define
\begin{equation}\label{eq:Lalpha-operators}
\mathcal L_{\alpha}(u):=|\nabla u|^{\alpha}\,F(D^2u),
\qquad
\mathcal L_{\alpha}^{\pm}(u):=|\nabla u|^{\alpha}\,\mathcal M^{\pm}_{\lambda,\Lambda}(D^2u).
\end{equation}
We will often compare \(\mathcal L_{\alpha}\) with its Pucci extremal counterparts \(\mathcal L_{\alpha}^{\pm}\).

Let \(\Omega\subset\mathbb R^n\) be open and \(f\in C(\Omega)\cap L^\infty(\Omega)\).
We adopt the notion of viscosity solution for 
\begin{equation}\label{eq:main}
|\nabla u|^{\alpha}\,F(D^2u)=f\quad \text{in }\Omega,
\end{equation}
introduced by Birindelli-Demengel in \cite[Definition 2.7]{Birindelli-Demengel2004} 
which tests only with touching functions whose gradient does not vanish at the contact point.

\begin{definition}\label{def:BD}
Let $f\in C(\Omega)$ and let $F:\mathcal S^n\to\mathbb R$ be $(\lambda,\Lambda)$-elliptic \Assump{assump:Pucci-ellipticity}, with $\alpha\ge0 $.
A continuous function $u$ on $\Omega$ is called a viscosity \emph{supersolution} (respectively \emph{subsolution}) of
\[
|\nabla u|^{\alpha}F(D^{2}u)=f(x)\quad\text{in }\Omega
\]
if for every $x_{0}\in\Omega$ one of the following conditions holds:
\begin{itemize}
    \item[\textit{(i)}] There exist $\delta>0$ and a constant $c\in\mathbb R$ such that
$u\equiv c$ in the open ball $B_{\delta}(x_{0})\subset\Omega$ and
\[
0\ \le\ f(x)\quad \text{for all }x\in B_{\delta}(x_{0})
\qquad
\big(\text{respectively } 0\ \ge\ f(x)\text{ for all }x\in B_{\delta}(x_{0})\big).
\]
    \item[\textit{(ii)}] For every $\varphi\in C^{2}(\Omega)$ such that
$u-\varphi$ attains a local minimum (respectively local maximum) at $x_{0}$ and $\nabla\varphi(x_{0})\neq 0$, one has
\[
|\nabla\varphi(x_{0})|^{\alpha}\,F\!\big(D^{2}\varphi(x_{0})\big)
\ \le\ f(x_{0})
\qquad
\big(\text{respectively } |\nabla\varphi(x_{0})|^{\alpha}\,F\!\big(D^{2}\varphi(x_{0})\big)
 \ \ge\ f(x_{0})\big).
\]
\end{itemize}
A viscosity \emph{solution} is a function that is both a subsolution and a supersolution.

\end{definition}

\begin{remark}\label{rem:BD-classical}
Since the operator $ \mathcal L_{\alpha}$ is continuous in the gradient variable and we are assuming $F(0)=0$,
Definition \ref{def:BD} is equivalent to the classical
definition of viscosity solutions of Crandall-Ishii-Lions
\cite{Crandall-Ishii-Lions1992},
as explicitly observed in \cite[Remark~2.2]{Birindelli-Demengel2010}.
\end{remark}

\begin{remark}[Rescaling while preserving ellipticity]\label{rem:rescaling}
	Let $u\in C(B_1)$ be a viscosity subsolution of
	\[
	|\nabla u|^{\alpha}F(D^2u)= f \qquad\text{in }B_1,
	\]
	where $F:\mathcal S^n\to\R$ is $(\lambda,\Lambda)$–elliptic and $f\in C(B_1)$.
	For $a,b>0$ define
	\[
	v(x):=a\,u(bx),\qquad x\in B_{1/b}.
	\]
	Then $v$ satisfies
	\[
	|\nabla v|^{\alpha}\,\widetilde F(D^2v)\le \widetilde f(x)\qquad\text{in }B_{1/b},
	\]
	with
	\[
	\widetilde F(M)=a b^{2} F\!\big(\tfrac{1}{a b^{2}}M\big),\qquad
	\widetilde f(x)=a^{\alpha+1} b^{\alpha+2} f(bx),
	\]
	and $\widetilde F$ is again $(\lambda,\Lambda)$–elliptic.
	
	In the special case $a=1/b$, one obtains the simpler rescaled equation
	\[
	|\nabla v|^{\alpha}\,\widetilde F(D^2v)\le \widetilde f(x)\qquad\text{in }B_{1/b},
	\]
	where now
	\[
	\widetilde F(M)=b F\!\big(\tfrac{1}{b}M\big),\qquad
	\widetilde f(x)=b f(bx).
	\]
\end{remark}

\textbf{Large gradient regime operators.}
In \cite{Imbert-Silvestre2016}, it is considered a notion of viscosity solutions for operators acting only when the gradient is large. This notion, along with the results obtained for these operators, is crucial in our analysis. We recall hereafter the definition and discuss its relation with Definition \ref{def:BD} in Lemma \ref{lem:BD-to-Pucci-LG}.

\begin{definition}\label{TERESA}
Let $\Omega\subset\mathbb R^{n}$ be an open set and $\gamma \ge0$. Let $f\in C(\Omega)$ and
$F:\mathbb R^{n}\times\mathcal S^{n}\to\mathbb R$ be continuous and monotone in $X$, that is $X\le Y $ implies $  F(p,X)\le F(p,Y)$ for all $p\in\mathbb R^n$ and $X,Y\in\mathcal S^n$.
We say that $u \in C(\Omega)$ is a viscosity \emph{supersolution} (respectively \emph{subsolution}) of 
\[
F(\nabla u,D^2 u) = f \qquad \text{in } \{x\in\Omega:\ |\nabla u(x)|\ge \gamma\}
\]
if, for every $x_0\in\Omega$ and every
$\varphi\in C^2(\Omega)$ such that $u-\varphi$ has a local minimum (respectively local maximum) at $x_0$ and
$|\nabla\varphi(x_0)|\ge \gamma$, then
\[
F\!\big(\nabla\varphi(x_0),D^2\varphi(x_0)\big)\ \ge\ f(x_0) \qquad
\big(\text{respectively } F\!\big(\nabla\varphi(x_0),D^2\varphi(x_0)\big)\ \le\ f(x_0) \big).
\]

\noindent We write $F(\nabla u,D^2u)=f$ in $\{x\in\Omega:\ |\nabla u(x)|\ge \gamma\}$ when both conditions above hold.
No requirement is imposed at contact points where $|\nabla\varphi(x_0)|<\gamma$. 
To ease the notation, whenever it is clear from the context, we call $\{x\in\Omega:\ |\nabla u(x)|\ge \gamma\}$ simply as $\{ |\nabla u|\ge \gamma\}$.
\end{definition}

We shall also introduce the related extremal operators.

\begin{definition}\label{def:IS-extremal}
\noindent Fix \(0<\lambda\le\Lambda\), a threshold \(\gamma>0\), and a parameter \(\beta\ge0\).
For \((p,X)\in\mathbb R^{n}\times\mathcal S^{n}\), set
\begin{equation}\label{eq:Pplus-gammabeta}
\mathcal P^{+}_{\gamma,\beta}(p,X):=
\begin{cases}
\Lambda\,\tr(X^{+})-\lambda\,\tr(X^{-})+\beta\,|p|, & \text{if } |p|\ge\gamma,\\[2pt]
+\infty, & \text{if } |p|<\gamma.
\end{cases}
\end{equation}
\begin{equation}\label{eq:Pminus-gammabeta}
\mathcal P^{-}_{\gamma,\beta}(p,X):=
\begin{cases}
\lambda\,\tr(X^{+})-\Lambda\,\tr(X^{-})-\beta\,|p|, & \text{if } |p|\ge\gamma,\\[2pt]
-\infty, & \text{if } |p|<\gamma.
\end{cases}
\end{equation}

\noindent 

\end{definition}
\begin{remark}\label{rem:LG-IS-equivalence}
When \(|p|\ge\gamma\) these extremal operators agree with the classical Pucci operators \textup{(}plus the first–order term \(\pm\,\beta|p|\)\textup{)}. Namely,
$\mathcal P^{\pm}_{\gamma,\beta}(p,X)=\mathcal M^{\pm}_{\lambda,\Lambda}(X)\pm\beta|p|$.
Outside the large–gradient region, for \(|p|<\gamma\), they provide no information. In our applications, we take \(\beta=0\).

 In particular, from Definition \ref{TERESA}, for every $u\in C(\Omega)$ and $f\in C(\Omega)$, one has that 
\[
\mathcal M^{-}_{\lambda,\Lambda}(D^{2}u)\le f \ \text{ in } \ \{ |\nabla u|\ge \gamma\}
\ \Longleftrightarrow \
u \text{ is a viscosity subsolution of } \mathcal P^{-}_{\gamma,0}(\nabla u,D^{2}u)= f,
\]
\[
\mathcal M^{+}_{\lambda,\Lambda}(D^{2}u)\ge -\,f \ \text{ in }\{|\nabla u|\ge \gamma\} \Longleftrightarrow \
u \text{ is a viscosity supersolution of } \mathcal P^{+}_{\gamma,0}(\nabla u,D^{2}u)= -\,f.
\]

\end{remark}

We now relate the notion of viscosity solution to \eqref{eq:main} given by Definition \ref{def:BD} with the one of Definition \ref{TERESA}.

\begin{lemma}\label{lem:BD-to-Pucci-LG}
	Let $\alpha\ge0$, $f\in C(\Omega)\cap L^{\infty}(\Omega)$, and let
	$F:\mathcal S^{n}\to\mathbb R$ be $(\lambda,\Lambda)$–elliptic with $F(0)=0$.
	Assume $u\in C(\Omega)$ is a viscosity supersolution (respectively subsolution) of
\eqref{eq:main} in the sense of Definition \ref{def:BD}.
	Then, for every threshold $\gamma>0$, the following inequality holds in the sense of Definition~\ref{TERESA} on the large–gradient set $\{|\nabla u|\ge\gamma\}$
    \[
    \displaystyle \mathcal M^-_{\lambda,\Lambda}(D^{2}u)\ \le\ \gamma^{-\alpha}\,\|f\|_{L^{\infty}(\Omega)}, \quad \big(\text{respectively }\mathcal M^+_{\lambda,\Lambda}(D^{2}u)\ \ge\ -\,\gamma^{-\alpha}\,\|f\|_{L^{\infty}(\Omega)}\big).
    \]
    
	In particular, if $u$ is a viscosity solution in the sense of Definition \ref{def:BD}, both inequalities hold simultaneously on $\{|\nabla u|\ge\gamma\}$.
\end{lemma}
\begin{proof}
\noindent Fix $\gamma>0$. By $(\lambda,\Lambda)$–ellipticity and $F(0)=0$,
\begin{equation}\label{eq:Pucci-bracket}
\mathcal M^-_{\lambda,\Lambda}(X)\ \le\ F(X)\ \le\ \mathcal M^+_{\lambda,\Lambda}(X)
\qquad \text{for all } \,X\in\mathcal S^n.
\end{equation}
We first consider the case when $u$ is a viscosity supersolution of \eqref{eq:main}.
Let $x_{0}\in\Omega$, if $u$ satisfies condition \textit{(i)} of Definition \ref{def:BD} then $\{|\nabla u| \geq \gamma \} \cap B_{\delta}(x_0) = \emptyset$. On the other hand, if 
$\varphi\in C^{2}(\Omega)$ is such that $u-\varphi$ has a local minimum at $x_{0}$
and $|\nabla\varphi(x_{0})|\ge\gamma$. 
By condition \textit{(ii)} of Definition~\ref{def:BD}, one has
\[
|\nabla\varphi(x_{0})|^{\alpha}\,F\!\big(D^{2}\varphi(x_{0})\big)\ \le\ f(x_{0})\ \le\ \|f\|_{L^{\infty}(\Omega)}.
\]
Hence
\[
F\!\big(D^{2}\varphi(x_{0})\big)\ \le\ |\nabla\varphi(x_{0})|^{-\alpha}\,\|f\|_{L^{\infty}(\Omega)}
\ \le\ \gamma^{-\alpha}\,\|f\|_{L^{\infty}(\Omega)}.
\]
Using \eqref{eq:Pucci-bracket},
\[
\mathcal M^-_{\lambda,\Lambda}\!\big(D^{2}\varphi(x_{0})\big)\ \le\ \gamma^{-\alpha}\,\|f\|_{L^{\infty}(\Omega)}.
\]

\noindent This is precisely the statement that
\[
\mathcal M^-_{\lambda,\Lambda}(D^{2}u)\ \le\ \gamma^{-\alpha}\,\|f\|_{L^{\infty}(\Omega)}
\quad \text{in } \{x\in\Omega:\ |\nabla u(x)|\ge \gamma\}.
\]

 A similar reasoning applies whenever $u$ is a viscosity subsolution of \eqref{eq:main}.
Indeed, let $x_{0}\in\Omega$ and let $\varphi\in C^{2}(\Omega)$ be such that $u-\varphi$ has a local maximum at $x_{0}$
and $|\nabla\varphi(x_{0})|\ge\gamma$.
By Definition~\ref{def:BD},
\[
|\nabla\varphi(x_{0})|^{\alpha}\,F\!\big(D^{2}\varphi(x_{0})\big)\ \ge\ f(x_{0})\ \ge\ -\,\|f\|_{L^{\infty}(\Omega)}.
\]
Thus
\[
F\!\big(D^{2}\varphi(x_{0})\big)\ \ge\ -\,|\nabla\varphi(x_{0})|^{-\alpha}\,\|f\|_{L^{\infty}(\Omega)}
\ \ge\ -\,\gamma^{-\alpha}\,\|f\|_{L^{\infty}(\Omega)}.
\]
By \eqref{eq:Pucci-bracket},
\[
\mathcal M^+_{\lambda,\Lambda}\!\big(D^{2}\varphi(x_{0})\big)\ \ge\ -\,\gamma^{-\alpha}\,\|f\|_{L^{\infty}(\Omega)},\]
which gives the desired result.

Combining the two cases yields the two large–gradient Pucci inequalities in the statement.
\end{proof}

\section{Construction of a barrier with geometry} \label{s:4}
In this section, we provide an explicit useful barrier with geometry that plays a key role in the proof of Theorem \ref{HOL-quantitative}.

\begin{proposition}[Existence, smoothness and geometry of barriers]\label{prop:barrier}
	For every \(M\geq 0\) and \(R>0\) there exists
	\[
	\Gamma \in C^\infty(\overline{\Aaa_R}), 
	\qquad \Aaa_R := B_R \setminus \overline{B_{R/2}},
	\]
	such that
	\[
	\begin{cases}
	|\nabla \Gamma|^ {\alpha}{\cM}^-_{\lambda,\Lambda}(D^2\Gamma) 
		\;\ge\; c_0\,\dfrac{M^{1+\alpha}}{R^{2+\alpha}} & \text{in } \Aaa_R,\\[6pt]
		\Gamma = 0 & \text{on } \partial B_R,\\
		\Gamma = M & \text{on } \partial B_{R/2},
	\end{cases}
	\]
	and, moreover, for every \(x\in \Aaa_R\),
	\[
	A_1\,\dfrac{M}{R}\,\mathrm{dist}(x,\partial B_R)
	\;\le\; \Gamma(x)
	\;\le\; A_2\,\dfrac{M}{R}\,\mathrm{dist}(x,\partial B_R),
	\qquad
	A_3\,\dfrac{M}{R}
	\;\le\; |\nabla \Gamma(x)|
	\;\le\; A_4\,\dfrac{M}{R}.
	\]
	
	Here the constants \(c_0,A_1,A_2,A_3,A_4>0\) depend only on \(n,\alpha,\lambda,\Lambda\).
\end{proposition}

\begin{proof}
	We begin with the case $M=R=1$.  
	Define the radial function
	\[
	\Gamma(x) = \frac{|x|^{-\beta}-1}{2^\beta-1}, 
	\qquad x\in \overline{\Aaa_1},
	\]
	where $\beta>1$ will be fixed later.  	Let $\varphi:[1/2,1]\to \R$ be the profile
	\[
	\varphi(t) = \frac{t^{-\beta}-1}{2^\beta-1}.
	\]
	Then $\Gamma(x)=\varphi(|x|)$, and direct differentiation gives
	\[
	\varphi'(t) = -\frac{\beta}{2^\beta-1}\,t^{-(\beta+1)}, 
	\qquad
	\varphi''(t) = \frac{\beta(\beta+1)}{2^\beta-1}\,t^{-(\beta+2)}.
	\]
	Consequently,
	\[
	\nabla \Gamma(x) = \varphi'(|x|)\frac{x}{|x|},
	\qquad
	D^2\Gamma(x) = \varphi''(|x|)\frac{x\otimes x}{|x|^2}
	+ \frac{\varphi'(|x|)}{|x|}\left(I-\frac{x\otimes x}{|x|^2}\right).
	\]
	Thus the eigenvalues of $D^2\Gamma(x)$ are
	\[
	\mu_1 = \varphi''(|x|) = \frac{\beta(\beta+1)}{2^\beta-1}\,|x|^{-(\beta+2)} > 0 
	\quad \text{ with multiplicity $1$},
	\]
	\[
	\mu_2 = \frac{\varphi'(|x|)}{|x|} = -\frac{\beta}{2^\beta-1}\,|x|^{-(\beta+2)} < 0
	\quad \text{with  multiplicity $n-1$}.
	\]
	
	Applying the Pucci operator,
	\[
	{\cM}^-_{\lambda,\Lambda}(D^2\Gamma(x)) 
	= \lambda \mu_1 + \Lambda(n-1)\mu_2
	= \frac{\beta |x|^{-(\beta+2)}}{2^\beta-1}
	\big( (\beta+1)\lambda - (n-1)\Lambda \big).
	\]
	Choosing
	\[
	\beta := \frac{(n-1)\Lambda}{\lambda}+2,
	\]
	we guarantee that $(\beta+1)\lambda-(n-1)\Lambda>0$, hence
	\[
	{\cM}^-_{\lambda,\Lambda}(D^2\Gamma(x)) \geq c_1 > 0,
	\qquad x\in \Aaa_1,
	\]
	with $c_1=c_1(n,\lambda,\Lambda)$.  
	Since
	\[
	|\nabla \Gamma(x)| 
	= \frac{\beta}{2^\beta-1}|x|^{-(\beta+1)},
	\]
	we conclude
	\[
	\mathcal{L}_\alpha^-(\Gamma)(x)
	= |\nabla \Gamma(x)|^\alpha \, {\cM}^-_{\lambda,\Lambda}(D^2\Gamma(x))
	\;\geq\;  \left(\frac{\beta}{2^{\beta}-1} \right)^\alpha c_1 \geq  c_0 > 0
	\quad \text{in } \Aaa_1,
	\]
	where $c_0=c_0(n,\lambda,\Lambda,\alpha)$.
	Next, we examine the geometry of $\Gamma$.  
	Since $\varphi(1)=0$ and $\varphi$ is convex, for $t\in[1/2,1]$, we have
	\[
	\varphi(t)\geq \varphi(1)+\varphi'(1)(t-1) 
	= -\varphi'(1)(1-t).
	\]
	Because $\varphi'(1)=-\tfrac{\beta}{2^\beta-1}<0$, this yields the explicit lower bound
	\[
	\varphi(t)\geq \frac{\beta}{2^\beta-1}(1-t).
	\]
	Thus, for $x\in \Aaa_1$,
	\[
	\Gamma(x)=\varphi(|x|) \geq 
	\frac{\beta}{2^\beta-1}\,\dist(x,\partial B_1).
	\]
	
	For the upper bound, by the mean value theorem there exists $\eta_t\in(t,1)$ such that
	\[
	\varphi(t) = -\varphi'(\eta_t)(1-t).
	\]
	Since
	\[
	-\varphi'(s)=\frac{\beta}{2^\beta-1}\,s^{-(\beta+1)},
	\qquad s\in[1/2,1],
	\]
	we deduce the two–sided bound
	\[
	\frac{\beta}{2^\beta-1}
	\ \leq\  -\varphi'(s)\ \leq\ \frac{\beta\,2^{\beta+1}}{2^\beta-1}.
	\]
	Hence
	\[
	\Gamma(x) \leq \frac{\beta\,2^{\beta+1}}{2^\beta-1}\,\dist(x,\partial B_1),
	\qquad x\in \Aaa_1.
	\]
	Combining the last two displays yields the desired geometric inequalities in the normalized case.
	
	As for the gradient,
	\[
	|\nabla \Gamma(x)| 
	= \frac{\beta}{2^\beta-1}|x|^{-(\beta+1)},
	\qquad x\in \Aaa_1,
	\]
	and since $|x|\in[1/2,1]$, we obtain
	\[
	\frac{\beta}{2^\beta-1}
	\;\leq\; |\nabla \Gamma(x)| \;\leq\; 
	\frac{\beta\,2^{\beta+1}}{2^\beta-1}.
	\]
	
	\smallskip
	Since $\Gamma$ is given by the composition of smooth functions in the annulus $\overline{\Aaa_1}$ we deduce that $\Gamma$ belongs to $C^\infty(\overline{\Aaa_1})$, and in particular,
	all derivatives extend continuously up to $\partial B_{1/2}$ and $\partial B_1$.
	
	We now scale back to arbitrary $M>0$, $R>0$.  
	Define
	\[
	\Gamma_{M,R}(x) := M \,\Gamma\!\left(\frac{x}{R}\right),
	\qquad x\in \overline{\Aaa_R}.
	\]
	By construction,
	\[
	\Gamma_{M,R}=0 \text{ on }\partial B_R,
	\qquad
	\Gamma_{M,R}=M \text{ on }\partial B_{R/2}.
	\]
	Moreover,
	\[
	\nabla \Gamma_{M,R}(x)=\frac{M}{R}\nabla \Gamma\left(\frac{x}{R}\right), 
	\qquad
	D^2\Gamma_{M,R}(x)=\frac{M}{R^2}D^2\Gamma\left(\frac{x}{R}\right),
	\]
	and therefore
	\[
	\mathcal{L}_\alpha^-(\Gamma_{M,R})(x)
	= \frac{M^{1+\alpha}}{R^{2+\alpha}}
	\left(\mathcal{L}_\alpha^-(\Gamma)\left(\frac{x}{R}\right) \right)
	\ \ge\ c_0 \frac{M^{1+\alpha}}{R^{2+\alpha}}.
	\]
	As the geometric estimates also scale accordingly, we obtain
	\[
	A_1\frac{M}{R}\dist(x,\partial B_R)
	\;\leq\; \Gamma_{M,R}(x) \;\leq\;
	A_2\frac{M}{R}\dist(x,\partial B_R),
	\qquad
	A_3 \frac{M}{R}
	\;\leq\; |\nabla \Gamma_{M,R}(x)|
	\;\leq\; A_4 \frac{M}{R}.
	\]
	Since $\Gamma\in C^\infty(\overline{\Aaa_1})$ and smoothness is preserved under scaling and multiplication, it follows that $\Gamma_{M,R}\in C^\infty(\overline{\Aaa_R})$.  
	This completes the proof.
\end{proof}


\section{Weak Harnack and Harnack Inequalities } \label{s:5}

This section aims to establish both the weak Harnack and the Harnack inequalities for equation \eqref{eq:main}.
The weak Harnack inequality is obtained by suitably combining the result of Imbert and Silvestre \cite{Imbert-Silvestre2016} for the large gradient regime with Lemma \ref{lem:BD-to-Pucci-LG}.
The full Harnack inequality follows a similar strategy. However, in this case, it is essential to carefully track the monotone dependence with respect to the threshold $\gamma$ for the Harnack constant.

\subsection{Weak Harnack Inequality}

We first restate the weak Harnack inequality obtained in \cite[Theorem 5.1]{Imbert-Silvestre2016} in a convenient way.

\begin{theorem}\label{thm:WH-LG-Pucci}
	Fix $n\ge2$ and $0<\lambda\le\Lambda$. There exist small constants $\varepsilon_0>0$ and $\varepsilon>0$, and a constant $C=C(n,\lambda,\Lambda)\ge1$, such that if $\gamma\le\varepsilon_0$ and $u\in C(B_2)$ satisfies, in the viscosity sense,
	\[
	u\ge0 \ \text{in } B_2,\qquad \mathcal M^-_{\lambda,\Lambda}(D^2u)\ \le\ 1 \ \ \text{in } \{x\in B_2:\ |\nabla u(x)|\ge\gamma\},
	\]
	and $\inf_{B_1}u\le 1$, then $\bigl|\{u>t\}\cap B_1\bigr|\le C\,t^{-\varepsilon}$ for all $t>0$. In particular, $\ \|u\|_{L^{\varepsilon/2}(B_1)}\le C$.
\end{theorem}

\begin{proof} 
 By Remark~\ref{rem:LG-IS-equivalence} with $\beta=0$, the hypotheses implies that
$u$ is a viscosity subsolution of
\[
\mathcal P^{-}_{\gamma,0}\big(\nabla u,D^{2}u\big)\ \le\ 1 \quad\text{in } B_{2}.
\]
Hence $u$ satisfies the assumptions of \cite[Thm.~5.1]{Imbert-Silvestre2016} (with threshold $\gamma\le\varepsilon_{0}$),
which yields
\[
\bigl|\{u>t\}\cap B_{1}\bigr|\ \le\ C\,t^{-\varepsilon}\qquad\forall\,t>0.
\]
Fix $\theta:=\varepsilon/2$. The layer–cake representation gives
\[
\int_{B_{1}}u^{\theta}
=\theta\!\int_{0}^{\infty} t^{\theta-1}\,|\{u>t\}\cap B_{1}|\,dt
\le |B_{1}|+\theta C\!\int_{1}^{\infty} t^{\theta-1-\varepsilon}\,dt
\le C,
\]
so $\|u\|_{L^{\varepsilon/2}(B_{1})}\le C$, as claimed.
\end{proof}

We are now ready to prove the weak Harnack inequality for $u$ that satisfies the following inequality
\begin{equation} \label{eq:supersolPucci}
    |\nabla u|^{\alpha}\mathcal M^-_{\lambda,\Lambda}(D^2u)\ \le\ \|f\|_{L^\infty(B_1)}\quad\text{in }B_1.
\end{equation}

\begin{theorem}[$L^{\varepsilon/2}$ estimate for $|\nabla u|^{\alpha}\mathcal M^-_{\lambda,\Lambda}(D^2u)$]\label{thm:WH}
	Fix $n\ge2$ and $0<\lambda\le\Lambda$, $\alpha>0$. There exist $\varepsilon=\varepsilon(n,\lambda,\Lambda)\in(0,1)$ and a constant $C_{\mathrm{WH}}=C_{\mathrm{WH}}(n,\lambda,\Lambda)\ge1$ such that, if $u\in C(B_1)$ is nonnegative and \eqref{eq:supersolPucci} holds in the sense of Definition \ref{def:BD},
	then
	\[
	\|u\|_{L^{\varepsilon/2}(B_{1/2})}\ \le\ C_{\mathrm{WH}}\Big(\inf_{B_{1/2}}u+\|f\|_{L^\infty(B_1)}^{\frac{1}{1+\alpha}}\Big).
	\]
\end{theorem}

\begin{proof} 
	Let us fix $\gamma>0$. Since $u$ solves \eqref{eq:supersolPucci} in the sense of Definition \ref{def:BD}, applying Lemma \ref{lem:BD-to-Pucci-LG} yields, in the viscosity sense of Definition~\ref{TERESA},
	\begin{equation}\label{eq:LG-reduction}
		\mathcal M^-_{\lambda,\Lambda}(D^2u)\ \le\ \gamma^{-\alpha}\,\|f\|_{L^\infty(B_1)}
		\qquad\text{in } \{|\nabla u|\ge\gamma\}\cap B_1 .
	\end{equation}
	Equivalently, by Remark~\ref{rem:LG-IS-equivalence}, we get
	\[
	\mathcal P^{-}_{\gamma,0}\!\big(\nabla u,D^2u\big)\ \le\ \gamma^{-\alpha}\,\|f\|_{L^\infty(B_1)} \quad\text{in } B_1.
	\]

	 Now, in order to reduce to the setting of Theorem \ref{thm:WH-LG-Pucci}, we define
	\[
	K:=\inf_{B_{1/2}}u+\|f\|_{L^\infty(B_1)}^{\frac{1}{1+\alpha}},\quad \quad  v:=\frac{u}{K}.
	\]
	Then $v\ge0$ in $B_1$ and $\inf_{B_{1/2}}v\le1$. From \eqref{eq:LG-reduction} and $\nabla v=\nabla u/K$, we obtain
	\begin{equation}\label{eq:v-ineq}
		\mathcal M^-_{\lambda,\Lambda}(D^2v)\ \le\ \frac{\gamma^{-\alpha}\,\|f\|_{L^\infty(B_1)}}{K}
		\qquad\text{in }\ \{|\nabla v|\ge \gamma/K\}\cap B_1 .
	\end{equation}

	Let $\varepsilon_0\in(0,1)$ be the small parameter in Theorem \ref{thm:WH-LG-Pucci}. Consider
	\[
	\gamma:=\min\Big\{\varepsilon_0\,K,\ \ \|f\|_{L^\infty(B_1)}^{\frac{1}{1+\alpha}}\Big\}.
	\]
	By definition, $\gamma/K\le\varepsilon_0$, and since $K\ge \|f\|_{L^\infty(B_1)}^{\frac{1}{1+\alpha}}$, we also have
	\[
	\frac{\gamma^{-\alpha}\,\|f\|_{L^\infty(B_1)}}{K}
	\ \le\
	\frac{\left(\|f\|_{L^\infty}^{\frac{1}{1+\alpha}}\right)^{-\alpha}\|f\|_{L^\infty}}{K}
	=\frac{\|f\|_{L^\infty}^{\frac{1}{1+\alpha}}}{K}\ \le\ 1.
	\]
	Thus, by \eqref{eq:v-ineq},
	\begin{equation}\label{eq:v-ready}
		\mathcal M^-_{\lambda,\Lambda}(D^2v)\ \le\ 1
		\quad\text{in }\ \{|\nabla v|\ge \gamma/K\}\cap B_1,
		\qquad \frac{\gamma}{K}\le\varepsilon_0,
		\qquad \inf_{B_{1/2}}v\le1 .
	\end{equation}

From \eqref{eq:v-ready}, $v$ satisfies the hypotheses of Theorem~\ref{thm:WH-LG-Pucci}. Hence there exist $\varepsilon=\varepsilon(n,\lambda,\Lambda)\in(0,1)$ and $C_{\mathrm{WH}}=C_{\mathrm{WH}}(n,\lambda,\Lambda)\ge1$ such that $\|v\|_{L^{\varepsilon/2}(B_{1/2})}\le C_{\mathrm{WH}}$. Since $u=K\,v$, it follows that
\[
\|u\|_{L^{\varepsilon/2}(B_{1/2})}
=K\,\|v\|_{L^{\varepsilon/2}(B_{1/2})}
\ \le\ C_{\mathrm{WH}}\Big(\inf_{B_{1/2}}u+\|f\|_{L^\infty(B_1)}^{\frac{1}{1+\alpha}}\Big),
\]
which is the desired estimate.

\end{proof}

\subsection{Harnack Inequality}

A key step in establishing the Harnack inequality for \eqref{eq:main} is the observation that the constant in the Harnack inequality of \cite[Theorem 1.3]{Imbert-Silvestre2016} depends monotonically on the threshold defining the large gradient regime. This fact is formalized in the following remark.

\begin{remark}\label{rem:monotone-r}
	Let $C_0>0$ and for $\tau>0$ define the admissible class
	\[
	\mathcal A(\tau):=\Bigl\{\,u\in C(\overline{B_1}),\ u\ge0:\ 
	\mathcal P^{-}_{\tau,0}(\nabla u,D^2u)\le C_0\ \text{and}\
	\mathcal P^{+}_{\tau,0}(\nabla u,D^2u)\ge -\,C_0\ \text{in }B_1\Bigr\},
	\]
	or equivalently, by Remark~\ref{rem:LG-IS-equivalence}, $u \in \Aaa(\tau)$ if and only if
	\[
	\mathcal M^{-}_{\lambda,\Lambda}(D^2u)\le C_0\ \ \text{and}\ \ 
	\mathcal M^{+}_{\lambda,\Lambda}(D^2u)\ge -\,C_0
	\quad\text{in } \{|\nabla u|\ge \tau\}\cap B_1.
	\]
	
\noindent	If $\tau_2\ge \tau_1$ and $u\in\mathcal A(\tau_1)$, then any test function with $|\nabla\varphi|\ge\tau_2$ also satisfies
	$|\nabla\varphi|\ge\tau_1$, so the defining viscosity inequalities hold there as well. Hence for $\tau_2\ge \tau_1$ one has $\mathcal A(\tau_1)\subseteq \mathcal A(\tau_2)$.
    
\noindent 	Define the set of admissible Harnack constants and the optimal one as
	\[
	\mathcal C_H(\tau):=\Bigl\{\,C>0:\ \sup_{B_{1/2}}u\le C\bigl(\inf_{B_{1/2}}u+C_0\bigr)\ 
	\text{for all }u\in\mathcal A(\tau)\Bigr\},
	\qquad
	\Phi_{n,\lambda,\Lambda}(\tau):=\inf_{C \in \mathcal C_H(\tau)}\, C .
	\]
	Since $\mathcal A(\tau_1)\subseteq \mathcal A(\tau_2)$ for $\tau_2\ge\tau_1$, the requirement
	“for all $u\in\mathcal A(\tau)$” becomes \emph{stronger} as $\tau$ increases, hence
	\[
	\mathcal C_H(\tau_2)\subseteq \mathcal C_H(\tau_1)\quad\Rightarrow\quad
	\Phi_{n,\lambda,\Lambda}(\tau_2)\ \ge\ \Phi_{n,\lambda,\Lambda}(\tau_1)\qquad \text{ for }\tau_2\ge\tau_1 .
	\]
	Thus $\Phi_{n,\lambda,\Lambda}$ is nondecreasing in $\tau$. 
\end{remark}

\begin{theorem}[Harnack inequality for degenerate elliptic operators] \label{Thm:HarnackBD}
	Fix $n\ge2$ and $0<\lambda\le\Lambda$. There exists a constant $C_{H}=C_{H}(n,\lambda,\Lambda)>0$ such that the following holds.
	Let $u\in C(\overline{B_1})$, $u\ge0$, and assume that, in the viscosity sense,
	\begin{equation}\label{eq:main-H}
		|\nabla u|^{\alpha}\,\cM^-_{\lambda,\Lambda}(D^2u)\le\|f\|_{L^\infty(B_1)}
		\quad\text{and}\quad
		|\nabla u|^{\alpha}\,\cM^+_{\lambda,\Lambda}(D^2u) \ge -\,\|f\|_{L^\infty(B_1)}
		\qquad\text{in } B_1.
	\end{equation}
	Then
	\begin{equation}\label{eq:OurHarnack-final}
		\sup_{B_{1/2}} u\ \le\
		C_{H}\Bigl(\inf_{B_{1/2}}u+\|f\|_{L^\infty(B_1)}^{\frac{1}{1+\alpha}}\Bigr).
	\end{equation}
\end{theorem}
\begin{proof}
We distinguish between the case when the right-hand side is zero and the case when it is not.

	\textit{Case \(\|f\|_{L^\infty(B_1)}=0\).}
	The equations \eqref{eq:main-H} reduce to
	\(
	|\nabla u|^{\alpha}\cM^-_{\lambda,\Lambda}(D^2u)\le0
	\) and
	\(
	|\nabla u|^{\alpha}\cM^+_{\lambda,\Lambda}(D^2u)\ge0
	\)
	in \(B_1\).
	By \cite[Lemma 6]{Imbert-Silvestre2013} we obtain the unweighted bounds
	\(
	\cM^-_{\lambda,\Lambda}(D^2u)\le0\le\cM^+_{\lambda,\Lambda}(D^2u)
	\)
	in \(B_1\), and the classical Harnack inequality \cite[Theorem~4.3 ]{Caffarelli-Cabre1995} yields
	\(
	\sup_{B_{1/2}}u\le C_H(n,\lambda,\Lambda)\,\inf_{B_{1/2}}u
	\),
	which is \eqref{eq:OurHarnack-final} when  the right hand side is \(0\).
	
	\medskip
\textit{Case \(\|f\|_{L^\infty(B_1)}>0\).}
	Let $\gamma>0$ to be chosen later. By Lemma~\ref{lem:BD-to-Pucci-LG}  and Remark~\ref{rem:LG-IS-equivalence}, we deduce that
	\[
	\mathcal P^{-}_{\gamma,0}(\nabla u,D^2u)\ \le\ g_\gamma,
	\qquad
	\mathcal P^{+}_{\gamma,0}(\nabla u,D^2u)\ \ge\ -\,g_\gamma
	\quad\text{in }B_1,
	\]
	with \(g_\gamma:=\gamma^{-\alpha}\|f\|_{L^\infty(B_1)}\).
	From the Harnack inequality of \cite[Thm.~1.3]{Imbert-Silvestre2016} we know that $\Phi_{n,\lambda,\Lambda}>-\infty$, where $\Phi_{n,\lambda,\Lambda}$ is defined in Remark \ref{rem:monotone-r}. This implies the existence of a nondecreasing function
	\(\Phi_{n,\lambda,\Lambda}:(0,\infty)\to[1,\infty)\) such that
	\begin{equation}\label{eq:Harnack-Phi}
		\sup_{B_{1/2}} u\ \le\ \Phi_{n,\lambda,\Lambda}(\tau)\,\Big(\inf_{B_{1/2}}u+g_\gamma\Big),
		\qquad \text{where }
		\tau:=\frac{\gamma}{\inf_{B_{1/2}}u+g_\gamma}.
	\end{equation}
	Now, we choose \(\gamma>0\) so that \(\gamma^{\alpha+1}=\|f\|_{L^\infty(B_1)}\).
	This gives that  \(g_\gamma=\gamma=\|f\|_{L^\infty(B_1)}^{\frac{1}{1+\alpha}}\) and therefore
	\[
	\tau=\frac{\gamma}{\inf_{B_{1/2}}u+\gamma}\ \le\ 1.
	\]
	From the monotonicity in \(\tau\) of $\Phi_{n,\lambda,\Lambda}$, \eqref{eq:Harnack-Phi} becomes
	\[
	\sup_{B_{1/2}} u\ \le\ \Phi_{n,\lambda,\Lambda}(1)\,\Big(\inf_{B_{1/2}}u+\gamma\Big).
	\]
	Setting \(C_H(n,\lambda,\Lambda):=\Phi_{n,\lambda,\Lambda}(1)\) and recalling \(\gamma=\|f\|_{L^\infty(B_1)}^{\frac{1}{1+\alpha}}\),
	we conclude
	\[
	\sup_{B_{1/2}} u\ \le\ C_H(n,\lambda,\Lambda)\,
	\Big(\inf_{B_{1/2}}u+\|f\|_{L^\infty(B_1)}^{\frac{1}{1+\alpha}}\Big),
	\]
	which is \eqref{eq:OurHarnack-final}. Note that once \(\tau\le1\) is fixed, the Harnack constant depends only on \(n,\lambda\) and \(\Lambda\).
	The parameter \(\alpha\) plays a role only in the preliminary choice of \(\gamma\).
\end{proof}

\section{Proof of the quantitative version of the Hopf Lemma} \label{s:6}

In this section we present the proof of Theorem~\ref{HOL-quantitative}, which relies mainly on the weak Harnack inequality established in Theorem~\ref{thm:WH} and on the barrier construction provided by 
Proposition \ref{prop:barrier}.

\begin{proof}[Proof of Theorem \ref{HOL-quantitative}] ~

\noindent\textbf{(A) Linear growth with respect to the distance from the boundary.}
	Let us fix the following notations
	\[
	\mathcal L^-_\alpha(u):=|\nabla u|^{\alpha}\,\mathcal M^-_{\lambda,\Lambda}(D^2u),\qquad
	\mathcal A_1:=B_1\setminus\overline{B_{1/2}},\qquad
	M:=\inf_{B_{1/2}}u\ \ge0.
	\]
	For nonnegative supersolutions of $\mathcal L^-_\alpha(u)\le f^{+}$ in $B_1$, from Theorem \ref{thm:WH}, there exist $C_{\mathrm{WH}}>0$ and $\varepsilon\in(0,1)$ (depending only on $n,\alpha,\lambda,\Lambda$) such that
	\begin{equation}\label{eq:UWH}
		\|u\|_{L^{\varepsilon}(B_{1/2})}\ \le\ C_{\mathrm{WH}}\Big(M+\|f^{+}\|_{L^{\infty}(B_1)}^{\frac{1}{1+\alpha}}\Big).
	\end{equation}
	In particular,
	\begin{equation}\label{eq:Mlower}
		M\ \ge\ C_{\mathrm{WH}}^{-1}\,\|u\|_{L^{\varepsilon}(B_{1/2})}\ -\ \|f^{+}\|_{L^{\infty}(B_1)}^{\frac{1}{1+\alpha}}.
	\end{equation}

	By Proposition~\ref{prop:barrier}, there exist universal constants $c_0>0$, $C_{\mathrm{bar}}>0$, $C_{\mathrm g}>0$ such that, for each $K>0$, one finds $\Gamma_K\in C^\infty(\overline{\mathcal A_1})$ with
	\[
	\Gamma_K=0\ \text{on }\partial B_1,\qquad \Gamma_K=K\ \text{on }\partial B_{1/2},\qquad \mathcal L^-_\alpha(\Gamma_K)\ \ge\ c_0\,K^{1+\alpha}\ \ \text{in }\mathcal A_1,
	\]
	and
	\begin{equation}\label{eq:geom}
		\Gamma_K(x)\ \ge\ C_{\mathrm{bar}}\,K\,\mathrm{dist}(x,\partial B_1),\qquad |\nabla\Gamma_K(x)|\ \ge\ C_{\mathrm g}\,K\quad \text{for all } x\in\overline{\mathcal A_1}.
	\end{equation}

Set $K:=M$ and $\Gamma_M:=\Gamma_K$. We split the proof into different cases. We begin by considering the relevant case in which the right-hand side is sufficiently small with respect to $M$. The complementary situation can be addressed straightforwardly.

\smallskip
\textit{Case 1:  $c_0\,M^{1+\alpha}\ >\ \|f^{+}\|_{L^\infty(B_1)}$.}
 Let us assume that 
\begin{equation}\label{eq:strict-gap-step4}
	c_0\,M^{1+\alpha}\ >\ \|f^{+}\|_{L^\infty(B_1)}.
\end{equation}
By Step~2, we have that $\mathcal L^-_\alpha(\Gamma_M)\ge c_0\,M^{1+\alpha}$. Moreover, the functions $u,\Gamma_M\in C(\overline{\mathcal A_1})$ satisfy
$\Gamma_M\le u$ on $\partial\mathcal A_1$ because
$\Gamma_M=M\le u$ on $\partial B_{1/2}$ and $\Gamma_M=0\le u$ on $\partial B_1$. We claim that $\Gamma_M\le u $ in $\mathcal A_1.$ Indeed, this follows from the comparison principle with strict inequalities on right–hand sides (see, e.g., \cite[Theorem 2.9]{Birindelli-Demengel2004}). Here, for the sake of completeness, we provide a direct viscosity argument. Suppose $\max_{\overline{\mathcal A_1}}(\Gamma_M - u)>0$
and let $x_{*}\in\mathcal A_1$ be an interior maximizer of
$w:=\Gamma_M-u$. Then $u-\Gamma_M$ has a local minimum at $x_{*}$ and,
since $u$ is a viscosity supersolution while $\Gamma_M\in C^\infty$,
\[
|\nabla \Gamma_M(x_*)|^{\alpha}\,
\cM^-_{\lambda,\Lambda}\big(D^{2}\Gamma_M(x_*)\big)
\ \le\ \|f^{+}\|_{L^\infty(B_1)}.
\]
On the other hand, by the barrier property and $M>0$,
\[
|\nabla \Gamma_M(x_*)|^{\alpha}\,
\cM^-_{\lambda,\Lambda}\big(D^{2}\Gamma_M(x_*)\big)
\ \ge\ c_0\,M^{1+\alpha},
\]
contradicting \eqref{eq:strict-gap-step4}. Therefore
\[
\Gamma_M\le u \qquad\text{in }\mathcal A_1.
\]
Now, since $u\ge M$ on $B_{1/2}$, by \eqref{eq:geom}, for any $x\in B_1$, we have
\begin{equation}\label{eq:geomglobal}
	u(x)\ \ge\ C_*\,M\,\mathrm{dist}(x,\partial B_1)
\end{equation}
where $C_*:=\min\{C_{\mathrm{bar}},1\}$. Then, by combining \eqref{eq:geomglobal} and \eqref{eq:Mlower}, we obtain
\begin{equation}\label{primo-bound}
u(x)\ \ge\ C_*\Big(C_{\mathrm{WH}}^{-1}\|u\|_{L^{\varepsilon}(B_{1/2})}
- \|f^{+}\|_{L^{\infty}(B_1)}^{\frac{1}{1+\alpha}}\Big)\,\mathrm{dist}(x,\partial B_1).
\end{equation}
By suitably enlarging the constants if necessary, we set
\[
A_1:=\frac{C_*}{C_{\mathrm{WH}}},\qquad \text{and}\qquad
A_2:=C_*\left( 1+ 2 c_0^{-\frac{1}{1+\alpha}} \right),
\]
so that the bound in \eqref{primo-bound} can be rewritten in a universal form as
\begin{equation}\label{vesr-star}
u(x)\ \ge\
\Big(A_1\|u\|_{L^{\varepsilon}(B_{1/2})}
- A_2\|f^{+}\|_{L^{\infty}(B_1)}^{\frac{1}{1+\alpha}}\Big)\,\mathrm{dist}(x,\partial B_1),
\end{equation}
for any $x\in B_1$.

\smallskip
\textit{Case 2:  $c_0\,M^{1+\alpha}\ \leq \ \|f^{+}\|_{L^\infty(B_1)}$.}
 On the other hand, if \eqref{eq:strict-gap-step4} does not hold, namely, when
\[
M \leq  c_0^{-\frac{1}{1+\alpha}} \| f^+ \|_{L^\infty(B_1)}^{\frac{1}{1+\alpha}},
\]
This implies that
\[
\begin{split}
M - 2 c_0^{-\frac{1}{1+\alpha}} \| f^+ \|_{L^\infty(B_1)}^{\frac{1}{1+\alpha}}
\leq c_0^{-\frac{1}{1+\alpha}} \| f^+ \|_{L^\infty(B_1)}^{\frac{1}{1+\alpha}} \leq 0.
\end{split}
\]
Then, straightforwardly, for all $x\in B_1$, it follows that
\[
u(x) \geq 0 \geq 
C_{*}\left( M -2 c_0^{-\frac{1}{1+\alpha}} \| f^+ \|_{L^\infty(B_1)}^{\frac{1}{1+\alpha}}\right) 
\, \mathrm{dist}(x, B_1),
\]
which together with \eqref{eq:Mlower} ensures
\begin{equation}\label{vers-tilde}
u(x)\geq  \left( C_{*} C_{WH}^{-1}
\| u \|_{L^\varepsilon(B_{1/2})}
-  C_{*}\left(1 + 2 c_0^{-\frac{1}{1+\alpha}} \right) \| f^+ \|_{L^\infty(B_1)}^{\frac{1}{1+\alpha}}
\right)
\mathrm{dist}(x,B_1).
\end{equation}
Therefore in any case combining \eqref{vesr-star} and \eqref{vers-tilde}, for any $x\in B_1$, we reach 
\begin{equation}\label{vers-fin}
u(x)\geq  \left(
A_1\| u \|_{L^\varepsilon(B_{1/2})}
- A_2 \| f^+ \|_{L^\infty(B_1)}^{\frac{1}{1+\alpha}}
\right)
\mathrm{dist}(x,B_1),
\end{equation}
which concludes the proof of (A).

	\medskip
	\noindent\textbf{(B) Control by the interior normal derivative.}
	Assume $u(x_0)=0$ for some $x_0\in\partial B_1$ and the interior normal derivative $\partial_\nu u(x_0)$ exists. 
    For $t>0$, let us set $x_t:=x_0+t\,\nu$. Since $\mathrm{dist}(x_t,\partial B_1)=t$, by applying \eqref{eq:HOL-1} at $x_t$, we get
	\begin{equation}\label{HOLonxt}
	\frac{u(x_t)}{t} \ge \left(A_1\,\|u\|_{L^{\varepsilon}(B_{1/2})}-A_2\,\|f^{+}\|_{L^\infty(B_1)}^{\frac{1}{1+\alpha}}\right).
	\end{equation}
	Letting $t\to 0$ in \eqref{HOLonxt}, we get
\[
\partial_\nu u(x_0)\ge A_1\,\|u\|_{L^{\varepsilon}(B_{1/2})}
- A_2\,\|f^{+}\|_{L^\infty(B_1)}^{\frac{1}{1+\alpha}}.
\]
Consequently,
\[
\|u\|_{L^{\varepsilon}(B_{1/2})} \le \frac{\partial_\nu u(x_0)}{A_1}
+ \frac{A_2}{A_1}\,\|f^{+}\|_{L^\infty(B_1)}^{\frac{1}{1+\alpha}},
\]
    which proves \eqref{eq:HOL-2} with $C:=A_2/A_1$.
	
	\medskip
	\noindent\textbf{(C) Supremum version under a two–sided Pucci inequality.}
	
Applying the interior Harnack inequality of Theorem \ref{Thm:HarnackBD} for nonnegative solutions of \eqref{eq:HOL-3}, and following the argument employed in points (A) and (B), now with 
$\sup_{B_{1/2}}u$ in place of $\|u\|_{L^{\varepsilon}(B_{1/2})}$, 
we deduce the existence of universal constants $A'_1, A'_2, C'>0$ such that    
	\[ u(x) \ge \left( A'_1\,\sup_{B_{1/2}}u-A'_2\,\|f^{+}\|_{L^\infty(B_1)}^{\frac{1}{1+\alpha}}\right)\,\mathrm{dist}(x,\partial B_1),
    \]
    for any $x\in B_1$, and 
    \[\sup_{B_{1/2}}u\le C'\left(\partial_\nu u(x_0)+\|f^{+}\|_{L^\infty(B_1)}^{\frac{1}{1+\alpha}}\right). \]
	
	\medskip
	\noindent\textbf{(D) Unique continuation at the boundary.}
	If $f\le0$ in $B_1$ (hence $f^{+}\equiv0$) and for some $x_0\in\partial B_1$ with $u(x_0)=0$, the interior normal derivative exists and equals zero at $x_0$, then by point (B) we obtain
\[
\|u\|_{L^\varepsilon(B_{1/2})} \le \frac{1}{A_1}\,\partial_\nu u(x_0)=0.
\]
Consequently, recalling that $u\in C(\overline{B_1})$, we conclude that $u\equiv0$ in $B_{1/2}$. Moreover, since $f^{+}\equiv0$, the viscosity inequality $|\nabla u|^{\alpha}\,\mathcal M^-_{\lambda,\Lambda}(D^2u)\le0$ reduces, by  \cite[Lemma 6]{Imbert-Silvestre2013}, to
\[
\mathcal M^{-}_{\lambda,\Lambda}(D^2u) \le 0 \quad \text{in } B_1
\]
in the viscosity sense. Thus, since $u\ge0$ in $B_1$ and vanishes on the nonempty open set $B_{1/2}$, the Strong Maximum Principle (see \cite[Proposition 4.9]{Caffarelli-Cabre1995}) yields $u\equiv0$ in $B_1$.

\end{proof}

\section{Failure of the Hopf Lemma in the large gradient regime}\label{s:7}

It is interesting to note that the Hopf–Oleinik Lemma fails to hold for equations defined in the large gradient regime.
The underlying reason is that such equations carry no information in regions where the gradient is small; consequently, tangential contact with the boundary is permitted.
This phenomenon is formalized in Proposition \ref{prop:hopf-fails-large-gradient-only}, whose proof is presented below.

\begin{proof}[Proof of Proposition \ref{prop:hopf-fails-large-gradient-only}]
	Let $\Upsilon(r)=(1-r)^2$ for $r\in[0,1]$ and set $u(x)=\Upsilon(|x|)$. Then
	$u\in C(\overline{B_1})\cap C^\infty(B_1\setminus\{0\})$. For $x\neq0$, denoting by $r=|x|$, we have
	\[
	\nabla u(x)=-2(1-r)\frac{x}{r},\qquad
	\Delta u(x)=\Upsilon''(r)+\frac{n-1}{r}\Upsilon'(r)=2n-\frac{2(n-1)}{r}.
	\]
    
	Throughout this proof we use Definition \ref{TERESA} with $F=\Delta$ and $L=1$:
	the supersolution condition is checked only at lower contact points with $|\nabla\phi|\ge1$, namely we fix $\gamma=1$. 
    
	First we will show that no $C^2$ test function can touch $u$ from below at the origin.
	Indeed, assume by contradiction that there exists $\phi\in C^2(B_1)$ such that $u-\phi$ has a local minimum at $0$.
	Fixed $e\in\mathbb S^{n-1}$, we set $x=te$. Taylor’s formula for $\phi$ at $0$ gives
\begin{equation}\label{phi(te)}
\phi(te)=\phi(0)+t\,\langle\nabla\phi(0),e\rangle+\tfrac12 t^2\, \langle D^2\phi(0)e, e\rangle+o(t^2),
	\end{equation}
	while
	\[
	u(te)=1-2|t|+t^2.
	\]
	Since $u-\phi$ has a local minimum at $0$, we have $\phi(te)\le u(te)$ whenever $|t|$ is sufficiently small. We treat the two situations $t>0$ and $t<0$ separately.
    
	\textit{(a) Case $t>0$.} Dividing by $t>0$ and using \eqref{phi(te)}, we get 
	\[
	\langle \nabla\phi(0), e\rangle +\frac12 t \langle D^2\phi(0)e,e\rangle +\frac{o(t^2)}{t} \le -2+t,
	\]
	which, letting $t\to 0$, implies
	\begin{equation}\label{eq:assurdo}
	\langle \nabla\phi(0), e \rangle \le -2.
	\end{equation}
	
	\emph{(b) Case $t<0$.} Let $s>0$ such that $t=-s<0$. Dividing by $s>0$ and using \eqref{phi(te)}, we obtain
\[
	-\langle \nabla\phi(0), e\rangle +\tfrac12 s\,\langle D^2\phi(0)e, e\rangle+\tfrac{o(s^2)}{s} \le -2+s.
\]
Letting $s\to 0$, this yields
\[
	\langle \nabla\phi(0),e\rangle \ge 2,
\]
which contradicts \eqref{eq:assurdo}.

Outside the origin $u$ is smooth and thus we can check directly the viscosity condition.
	Let $x_0\in B_1\setminus\{0\}$ and $\phi\in C^2(B_1)$ be such that $u-\phi$ has a local minimum at $x_0$. Since $u$ is smooth on $B_1\setminus\{0\}$ and $\phi\in C^2(B_1)$, the following contact relations hold
	\[
	\nabla\phi(x_0)=\nabla u(x_0)\quad \text{and}\quad D^2\phi(x_0)\ \le\ D^2u(x_0).
	\]
	Consequently,
	\[
	|\nabla\phi(x_0)|=|\nabla u(x_0)|=2(1-|x_0|).
	\]
	By Definition \ref{TERESA}, we need only consider the case $|\nabla\phi(x_0)|\ge1$, i.e. $|x_0|\le\tfrac12$.
	For such $x_0$,
	\[
	\Delta u(x_0)=2n-\frac{2(n-1)}{|x_0|} \le 2n-4(n-1)\ =\ 4-2n \le 0,
	\]
	which, taking traces in $D^2\phi(x_0)\le D^2u(x_0)$, yields
	\[
	\Delta\phi(x_0) \le\ \Delta u(x_0) \le 0.
	\]
	Thus, the supersolution requirement is satisfied at every lower contact point with $|\nabla\phi(x_0)|\ge1$, where the equality case $|\nabla\phi(x_0)|=1$ corresponds to $|x_0|=\frac12$.

	We conclude by checking that $u$ has a tangential behavior at the boundary.
    In fact, for any $\xi\in\partial B_1$ with outward unit normal $\nu_{\rm out}(\xi)=\xi$, we have 
	\[
	\partial_{\nu}^{\rm int}u(\xi)
	=\lim_{t\to0}\frac{u(\xi-t\nu_{\rm out})-u(\xi)}{t}
	=\lim_{t\to0}\frac{(1-(1-t))^2}{t}
	=\lim_{t\to0}\frac{t^2}{t}=0.
	\]
	Hence $u\ge0$ in $B_1$, $u=0$ on $\partial B_1$, and the interior normal derivative at the boundary minimum vanishes. Therefore, Hopf’s boundary point conclusion fails for this large–gradient regime.
\end{proof}

\section{Gradient estimates for solutions to $|\nabla u|^{\alpha} F(D^2 u)=f$} \label{s:8}

We now state a $C^{1,\gamma}$–regularity result for viscosity solutions of the degenerate equation 
\[
|\nabla u|^{\alpha}F(D^2u)=f,
\]
which combines the interior estimates of Imbert–Silvestre in \cite{Imbert-Silvestre2013} with the quantitative Hopf–Oleinik principle established in Theorem \ref{HOL-quantitative}. In particular, we are interested in pointwise gradient estimates. These provide a key tool for the analysis of the free boundary problems addressed in  Theorem~\ref{thm:Lip-FBP} and Theorem~\ref{thm:flame-deg}.

\begin{theorem}\label{thm:gradient-est}
	Assume that $\alpha \ge 0$, $F$ is uniformly elliptic with $F(0)=0$, and $f \in C(B_1)\cap L^\infty(B_1)$. Suppose that $u \in C(B_1)$ is a viscosity solution of $|\nabla u|^{\alpha} F(D^2 u)=f$ in $B_1$. Then there exist $\gamma\in (0,1]$ such that $u\in C_{loc}^{1,\gamma}(B_1)$ and $C>0$ such that
	\begin{equation}\label{eq:C1gamma}
		\|u\|_{C^{1,\gamma}(B_{1/2})}\;\le\; C\Big(\|u\|_{L^\infty(B_1)}+\|f\|_{L^\infty(B_1)}^{\frac{1}{1+\alpha}}\Big).
	\end{equation}
	Moreover, if in addition $u\ge 0$ in $B_1$, then
	\begin{equation}\label{eq:grad0}
		|\nabla u(0)|\;\le\; C\Big(u(0)+\|f\|_{L^\infty(B_1)}^{\frac{1}{1+\alpha}}\Big).
	\end{equation}
	Furthermore, in the case $0\le u \in C(\overline{B_1})$ and there exists $x_0\in\partial B_1$ with $u(x_0)=0$ and the inward unit normal derivative $\partial_\nu u(x_0)$ well defined $($where for the unit ball $\nu=-x_0)$, one has
	\begin{equation}\label{eq:grad0-bdry}
		|\nabla u(0)|\;\le\; C\Big(\partial_\nu u(x_0)+\|f\|_{L^\infty(B_1)}^{\frac{1}{1+\alpha}}\Big).
	\end{equation}
	All the constants above depend only on $\alpha$, the ellipticity constants of $F$, and the dimension $n$.
\end{theorem}

\begin{proof}
	Changing appropriately the constants in \cite[Theorem~1]{Imbert-Silvestre2013}, there exist $\gamma\in(0,1]$
	and a constant $C$, depending only on $n,\lambda,\Lambda$ and $\alpha$, such that
	\[
	[u]_{1+\gamma,B_{3/4}}\ \le\ C\Big(\|u\|_{L^\infty(B_1)}+\|f\|_{L^\infty(B_1)}^{\frac{1}{1+\alpha}}\Big).
	\]
Then, by Proposition~\ref{prop:campanato-uniform-taylor}, for every $x_0\in B_{1/2}$, there exists an affine map
	$\ell_{x_0}(x)=u(x_0)+p(x_0)\cdot(x-x_0)$ such that
	\[
	|u(x)-\ell_{x_0}(x)| \le T\,|x-x_0|^{1+\gamma}\qquad\text{for all }x\in B_{1/8}(x_0),
	\]
	where  $T:= C(\|u\|_{L^\infty(B_1)}+\|f\|_{L^\infty(B_1)}^{\frac{1}{1+\alpha}})$, for some constant $C=C(n,\lambda,\Lambda,\alpha)$. By  Theorem~\ref{thm:UTE-C1omega}, in the power–modulus case \(\omega(t)=t^\gamma\)
	(see also  Remark~\ref{power-modulus}), up to renaming constants, we conclude that
	\[
	\|u\|_{C^{1,\gamma}(B_{1/2})}
	\le C\Big(\|u\|_{L^\infty(B_1)}+\|f\|_{L^\infty(B_1)}^{\frac{1}{1+\alpha}}\Big),
	\]
    for some constant $C>0$ depending only on $n,\lambda, \Lambda$ and $\alpha$, which proves \eqref{eq:C1gamma}.

	Assume now that \(u\ge0\) in \(B_1\). By applying \eqref{eq:C1gamma} in \(B_{1/4}\), we get 
	\begin{equation}\label{useC1,a}
	|\nabla u(0)| \le C\Big(\sup_{B_{1/2}}u+\|f\|_{L^\infty(B_1)}^{\frac{1}{1+\alpha}}\Big).
	\end{equation}
	Moreover, the Harnack inequality \eqref{eq:OurHarnack-final} gives
    \begin{equation}\label{useHar}
	\sup_{B_{1/2}} u \le C\Big(u(0)+\|f\|_{L^\infty(B_1)}^{\frac{1}{1+\alpha}}\Big),
	\end{equation}
	Now, combing \eqref{useC1,a} and \eqref{useHar}, we obtain
	\[
	|\nabla u(0)|\ \le\ C\Big(u(0)+\|f\|_{L^\infty(B_1)}^{\frac{1}{1+\alpha}}\Big),
	\]
	which proves \eqref{eq:grad0}.
	
	For the boundary statement, let \(0\le u\in C(\overline{B_1})\), assume \(u(x_0)=0\) for some
	\(x_0\in\partial B_1\), and that the inward normal derivative \(\partial_\nu u(x_0)\) exists.
	Since $F$ is uniformly elliptic and $F(0)=0$, we can apply point (B) and (C) of Theorem~\ref{HOL-quantitative} which ensure that 
	\begin{equation}\label{useHOL}
	u(0)\ \le\ \sup_{B_{1/2}} u\ \le\ C\Big(\partial_\nu u(x_0)+\|f\|_{L^\infty(B_1)}^{\frac{1}{1+\alpha}}\Big).
	\end{equation}
	Finally, by substituting \eqref{useHOL} into \eqref{eq:grad0}, we get
	\[
	|\nabla u(0)|\ \le\ C\Big(\partial_\nu u(x_0)+\|f\|_{L^\infty(B_1)}^{\frac{1}{1+\alpha}}\Big),
	\]
	which is \eqref{eq:grad0-bdry}.
\end{proof}


%

\section{Gluing Sobolev functions across a rough interface}\label{s:glue-SVA}

In this section, we establish a gluing result for functions that vanish identically on a rough interface. Such a result plays a key role in applications to free boundary problems.
For this purpose, we require the functions to be continuous and to vanish identically along the interface. However, no assumption is needed regarding the regularity of the interface.

\begin{proposition}\label{thm:glue-SVA}
	Let $1\le p\le\infty$.  
	Let $U\subset\mathbb{R}^n$ be open, $V\subset U$ open, and let $A,B\subset U$ be disjoint open sets with an interface $\Gamma\subset U$ such that
	\[
	U\setminus\Gamma=A\cup B,\qquad
	\partial A\cap U\subset\Gamma,\qquad
	\partial B\cap U\subset\Gamma.
	\]
	Assume
	\begin{equation}\label{ass:SVA}
		u\in W^{1,p}(A)\cap C(\overline A\cap U),\qquad
		v\in W^{1,p}(B)\cap C(\overline B\cap U),\qquad
		u|_\Gamma=v|_\Gamma=0.
	\end{equation}
	Define $w:V\to\mathbb{R}$ by
	\[
	w=
	\begin{cases}
		u &\text{on }A\cap V,\\
		v &\text{on }B\cap V,\\
		0 &\text{on }\Gamma\cap V.
	\end{cases}
	\]
	Then $w\in W^{1,p}(V)\cap C(V)$,
	\[
	\nabla w=\chi_A\nabla u+\chi_B\nabla v\quad\text{a.e. in }V,
	\]
	and
	\[
	\|w\|_{L^p(V)}\le \|u\|_{L^p(A)}+\|v\|_{L^p(B)},\qquad
	\|\nabla w\|_{L^p(V)}\le \|\nabla u\|_{L^p(A)}+\|\nabla v\|_{L^p(B)},
	\]
	with the natural interpretation when $p=\infty$.
\end{proposition}

\begin{proof}
	We split the proof into three steps.
    
\smallskip
\noindent\textbf{Step 1: case $V\Subset U$.}
	Let $K$ be a compact set such that  $\overline V\subset K\Subset U$.  
	By \eqref{ass:SVA} and the uniform continuity of $u$ and $v$ on $K\cap\overline A$ and $K\cap\overline B$,  
	for every $\eta>0$ there exists $\delta_\eta\in(0,1]$ such that
	\[
	|u|<\eta\ \text{on }U_\eta\cap A,\qquad |v|<\eta\ \text{on }U_\eta\cap B,
	\]
	where \(U_\eta:=\{x\in K:\operatorname{dist}(x,\Gamma)<\delta_\eta\}\).
	Let us consider the piecewise linear truncation given by
	\[
	\Phi_\eta(t):=
	\begin{cases}
		t-\eta,& t>\eta,\\[2pt]
		0,& |t|\le \eta,\\[2pt]
		t+\eta,& t<-\eta.
	\end{cases}
	\]
	Then $\Phi_\eta$ is $1$–Lipschitz on $\R$ and belongs to $W^{1,\infty}(\mathbb R)$, with
	\begin{equation}\label{der-phi}
	\Phi_\eta'(t)=\chi_{\{|t|>\eta\}}\quad\text{a.e.}
	\end{equation}
	Let us define  $u_\eta:=\Phi_\eta(u)$ on $A$ and $v_\eta:=\Phi_\eta(v)$ on $B$.  
	By the Sobolev chain rule, we know that $ u_\eta\in W^{1,p}(A)$ and $v_\eta\in W^{1,p}(B)$, with
    \[\
	\nabla u_\eta=\Phi_\eta'(u)\,\nabla u\qquad \text{and}\qquad \nabla v_\eta=\Phi_\eta'(v)\,\nabla v\ \ \text{a.e.}
	\]
	Since $|\Phi_\eta(t)|\le|t|$ and $|\Phi_\eta'(t)|\le1$, we have also
	\begin{equation}\label{dis-norme}
\|u_\eta\|_{L^p(A)}\le \|u\|_{L^p(A)},\qquad
	\|v_\eta\|_{L^p(B)}\le \|v\|_{L^p(B)}\
	\end{equation}
and
\begin{equation}\label{dis-norme-grad}
\|\nabla u_\eta\|_{L^p(A)}\le \|\nabla u\|_{L^p(A)}\qquad \|\nabla v_\eta\|_{L^p(B)}\le \|\nabla v\|_{L^p(B)}.
\end{equation}
Moreover, by definition, $u_\eta=v_\eta=0$ on $U_\eta$.

	Now, setting $d(x):=\operatorname{dist}(x,\Gamma)$, for $s\in(0,\delta_\eta)$, let us define
	\[
	\zeta_s(t):=
	\begin{cases}
		0,& t\le \tfrac{s}{2},\\[2pt]
		\dfrac{2}{s}\,t-1,& \tfrac{s}{2}< t< s,\\[2pt]
		1,& t\ge s.
	\end{cases}
	\]
	Then $\zeta_s\in W^{1,\infty}(\mathbb R)$ with
	\[
	(\zeta_s)'(t)=\tfrac{2}{s}\,\chi_{(s/2,s)}(t)\quad\text{a.e.}\qquad \text{and}\qquad \|(\zeta_s)'\|_{L^\infty([0,s])}=\tfrac{2}{s}.
	\]
	  Let us define $\psi_s(x):=\zeta_s(d(x))$. Since $d$ is $1$–Lipschitz, we have
	\[
	\nabla\psi_s(x)=\tfrac{2}{s}\,\chi_{\{s/2<d(x)<s\}}\,\nabla d(x)
	\quad\text{a.e.}
	\]
	Thus $|\nabla\psi_s|\le 2 s^{-1}\chi_{L_s}$ with $L_s:=\{x\in V:s/2<d(x)<s\}\subset U_\eta$ and
	\[
	\psi_s=0\ \text{on }\{d\le s/2\}\cap V,\qquad
	\psi_s=1\ \text{on }\{d\ge s\}\cap V.
	\]
	
	Finally, we define
	\[
	w_\eta(x):=
	\begin{cases}
		u_\eta(x),&x\in A,\\
		v_\eta(x),&x\in B,\\
		0,&x\in \Gamma.
	\end{cases}
	\]
	 and, for $i\in\{1,\dots,n\}$ and $\varphi\in C_c^\infty(V)$, we set
	\[
	K_s:=\overline V\cap\{d\ge s/2\}\cap\supp\varphi,\qquad
	A_s:=A\cap K_s,\ \ B_s:=B\cap K_s.
	\]
	Since  $\operatorname{dist}(K_s,\Gamma)\ge s/2$ and $K_s\Subset U$, there exists
	\(\rho>0\) (e.g. $\rho:=\min\{\tfrac{s}{4},\tfrac12\operatorname{dist}(K_s,\partial U)\}\))
	such that the $\rho$–neighbourhood of $K_s$ lies in $U\setminus\Gamma$.  
	Furthermore, since $A,B$ are open and cover $U\setminus\Gamma$, we get $\overline{A_s}\Subset A$ and $\overline{B_s}\Subset B$. In particular, we have
	\[
	\psi_s\varphi\chi_A\in W^{1,\infty}_0(A),\qquad
	\psi_s\varphi\chi_B\in W^{1,\infty}_0(B).
	\]
	Now, by testing $u_\eta$ on $A$ and $v_\eta$ on $B$ with these functions, we obatian
	\begin{equation}\label{integral-1}
	\int_{A} u_\eta\,\partial_i(\psi_s\varphi)
	= -\int_{A} \partial_i u_\eta\,\psi_s\varphi,\qquad
	\int_{B} v_\eta\,\partial_i(\psi_s\varphi)
	= -\int_{B} \partial_i v_\eta\,\psi_s\varphi.
	\end{equation}
	On the other hand, since $\supp\nabla\psi_s\subset L_s\subset U_\eta$ and  $u_\eta=v_\eta=0$ on $U_\eta$, we have
	\[
	\int_A u_\eta\,\varphi\,\partial_i\psi_s=0,\qquad
	\int_B v_\eta\,\varphi\,\partial_i\psi_s=0.
	\]
	Hence, expanding $\partial_i(\psi_s\varphi)$ and using \eqref{integral-1}, we conclude 
	\[
	\int_A u_\eta\,\psi_s\,\partial_i\varphi
	= -\int_A \partial_i u_\eta\,\psi_s\varphi,\qquad
	\int_B v_\eta\,\psi_s\,\partial_i\varphi
	= -\int_B \partial_i v_\eta\,\psi_s\varphi.
	\]
	Now, summing, using $w_\eta=u_\eta$ on $A$ and $w_\eta=v_\eta$ on $B$ and recalling that $\varphi\in C_c^\infty(V)$, we get
	\[
	\int_V w_\eta\,\psi_s\,\partial_i\varphi
	= -\int_V \bigl(\chi_A\partial_i u_\eta+\chi_B\partial_i v_\eta\bigr)\,\psi_s\,\varphi.
	\]
	Then, letting $s\to 0$, by the Dominated Convergence Theorem, we conclude
	\begin{equation}\label{final-weta}
	\int_V w_\eta\,\partial_i\varphi
	= -\int_V \bigl(\chi_A\partial_i u_\eta+\chi_B\partial_i v_\eta\bigr)\,\varphi,
	\end{equation}
	which shows $\nabla w_\eta=\chi_A\nabla u_\eta+\chi_B\nabla v_\eta$ a.e. in $V$.

    \smallskip
    \noindent\textbf{Step 2: Estimates and limit passage.}
	By definition and Step 1, we have
	\[
	w_\eta=u_\eta\chi_A+v_\eta\chi_B,\qquad
	\nabla w_\eta=\chi_A\nabla u_\eta+\chi_B\nabla v_\eta\ \ \text{a.e. in }V.
	\]
	Thus, recalling \eqref{dis-norme} and \eqref{dis-norme-grad}, we have
	\begin{equation}\label{bound-w-1}
	\|w_\eta\|_{L^p(V)}
	\le \|u_\eta\|_{L^p(A)}+\|v_\eta\|_{L^p(B)}
	\le \|u\|_{L^p(A)}+\|v\|_{L^p(B)}
	\end{equation}
    and
	\begin{equation}\label{bound-w-2}
	\|\nabla w_\eta\|_{L^p(V)}
	\le \|\nabla u_\eta\|_{L^p(A)}+\|\nabla v_\eta\|_{L^p(B)}
	\le \|\nabla u\|_{L^p(A)}+\|\nabla v\|_{L^p(B)}.
	\end{equation}
	For $1\le p<\infty$, since $\Phi_\eta(t)\to t$ as $\eta\to 0$ and $|\Phi_\eta(t)-t|\le2|t|$, by the Dominated Convergence Theorem, we conclude
	\[
	\|u_\eta-u\|_{L^p(A)}^p\to0,\qquad
	\|v_\eta-v\|_{L^p(B)}^p\to0,
	\]
    as $\eta\to 0$. So setting $w := u\chi_A + v\chi_B$, we have
	\[
	\|w_\eta-w\|_{L^p(V)}
	\le \|u_\eta-u\|_{L^p(A)}+\|v_\eta-v\|_{L^p(B)}\to0,\quad \text{as }\ \eta \to0.
	\]
	Moreover, recalling \eqref{der-phi}, we also have
	\[
	\|\nabla u_\eta-\nabla u\|_{L^p(A)}^p
	\le 2^p\!\int_A |\nabla u|^p\,\chi_{\{|u|\le\eta\}} \to 0,\qquad \|\nabla v_\eta-\nabla v\|_{L^p(B)}^p
	\le 2^p\!\int_B |\nabla v |^p\,\chi_{\{|u|\le\eta\}} \to 0,	\]
	and similarly for $v_\eta$. Hence
	\[
	\nabla w_\eta\to \chi_A\nabla u+\chi_B\nabla v\quad\text{in }L^p(V).
	\]
	When $p=\infty$, the uniform bounds in \eqref{bound-w-1} and \eqref{bound-w-2} ensure
	\[
	\sup_\eta\|w_\eta\|_{W^{1,\infty}(V)}<\infty.
	\]
	Then by Banach–Alaoglu Theorem, there exist $w^{(\infty)}\in L^\infty(V)$ and $G^{(\infty)}\in L^\infty(V;\mathbb R^n)$ such that
	\[
	w_\eta \stackrel{*}{\rightharpoonup} w^{(\infty)}\ \text{in }L^\infty(V),\qquad
	\nabla w_\eta \stackrel{*}{\rightharpoonup} G^{(\infty)}\ \text{in }L^\infty(V;\mathbb R^n).
	\]
	By \eqref{final-weta}, we have
\[
\int_V w_\eta\,\partial_i\varphi
= -\int_V \bigl(\chi_A\,\Phi_\eta'(u)\,\partial_i u
+\chi_B\,\Phi_\eta'(v)\,\partial_i v\bigr)\,\varphi.
\]
Since $\Phi_\eta'(\cdot)\to1$ a.e. and $0\le\Phi_\eta'\le1$, passing to the limit yields
\[
\int_V w^{(\infty)}\,\partial_i\varphi
= -\int_V (\chi_A\,\partial_i u+\chi_B\,\partial_i v)\,\varphi.
\]
Hence $w^{(\infty)}=w$ and $G^{(\infty)}=\chi_A\nabla u+\chi_B\nabla v$. In particular,
\[
w_\eta \stackrel{*}{\rightharpoonup} w
\quad\text{and}\quad
\nabla w_\eta \stackrel{*}{\rightharpoonup} \chi_A\nabla u + \chi_B\nabla v
\qquad\text{in }L^\infty(V).
\]

    \smallskip
	\noindent\textbf{Step 3: general $V\subset U$.}
	Let $\varphi\in C_c^\infty(V)$ and set $K:=\supp\varphi$.  
	Since $K\subset V\subset U$ is compact, there exists an open set $V_\varphi$ such that
	\[
	K \subset V_\varphi \Subset U\qquad\text{and}\qquad V_\varphi\subset V.
	\]
	By Step 1 applied to $V_\varphi\Subset U $, we have that for every $i\in\{1,\dots,n\}$,
	\begin{equation}\label{int-V-phi}
	\int_{V_\varphi} w\,\partial_i\varphi
	= -\int_{V_\varphi} \bigl(\chi_A\,\partial_i u+\chi_B\,\partial_i v\bigr)\,\varphi .
	\end{equation}
	However, since $\varphi$ vanishes outside $K\subset V_\varphi$, the same equality in \eqref{int-V-phi} holds with $V$ in place of $V_\varphi$, namely
	\[
	\int_{V} w\,\partial_i\varphi
	= -\int_{V} \bigl(\chi_A\,\partial_i u+\chi_B\,\partial_i v\bigr)\,\varphi .
	\]
	Since $\varphi\in C_c^\infty(V)$ is arbitrary, we conclude 
	\[
	w\in W^{1,p}(V)\qquad \text{and}\qquad
	\nabla w=\chi_A\nabla u+\chi_B\nabla v\quad\text{a.e. in }V.
	\]
	To obtain global $L^p$ estimates, let us consider an increasing exhaustion of $V$, i.e. a sequence of open sets $\{V_k\}_{k\in\mathbb{N}}$, such that  $V_k\Subset U$, $V_k\subset V_{k+1}$, for any $n\in \N $ and $\bigcup_{k \in  \N }V_k=V$.  
	By applying estimates \eqref{dis-norme} and \eqref{dis-norme-grad} on  $V_k$ and letting $k\to\infty$, we get
	\[
	\|w\|_{L^p(V)}\le \|u\|_{L^p(A)}+\|v\|_{L^p(B)},\qquad
	\|\nabla w\|_{L^p(V)}\le \|\nabla u\|_{L^p(A)}+\|\nabla v\|_{L^p(B)}
	\]
	for every $1\le p\le\infty$.
	Finally, since $u$ and $v$ are continuous up to $\Gamma$ and vanish there,
	$w=u\chi_A+v\chi_B$ is continuous in $V$.
\end{proof}

As a direct application of Proposition~\ref{thm:glue-SVA}, we obtain the following result concerning the positive part of a Sobolev function.
 
\begin{corollary}[Regularity of Positive part on subdomains via gluing]\label{thm:positive-part-generic}
	\noindent Let $1\le p\le\infty$. Let $U\subset\mathbb{R}^n$ be open and let $V\subset U$ be open. Let $u\in C(U)$ and set
	\[
	A_U:=\{x\in U:\ u(x)>0\},\qquad A_V:=A_U\cap V.
	\]
	Assume $u\in W^{1,p}(A_U)$. Then the positive part $u^{+}:=\max\{u,0\}$ belongs to $W^{1,p}(V)$ and
	\begin{equation}\label{eq:pos-Lp-generic}
		\|u^{+}\|_{L^{p}(V)}\le \|u\|_{L^{p}(A_V)},\qquad
		\|\nabla u^{+}\|_{L^{p}(V)}\le \|\nabla u\|_{L^{p}(A_V)}.
	\end{equation}
	Moreover,
	\[
	\nabla u^+=\chi_{A_U}\,\nabla u\quad\text{a.e. in }V,
	\]
	and, for $1\le p<\infty$,
	\[
	\|u^{+}\|_{L^{p}(V)}=\|u\|_{L^{p}(A_V)},\qquad
	\|\nabla u^{+}\|_{L^{p}(V)}=\|\nabla u\|_{L^{p}(A_V)}.
	\]
\end{corollary}

\begin{proof}
	\noindent Let us fix the sets
	\[
	A:=\{u>0\}\cap U,\qquad B:=\{u<0\}\cap U,\qquad \Gamma:=\{u=0\}\cap U.
	\]
	By continuity, $u=0$ on $\Gamma$, and by assumptions $u\in W^{1,p}(A)$. Now, let us set $v\equiv 0$ on $B$, so $v\in W^{1,p}(B)$ and $v=0$ on $\Gamma$. Then the sets $A$ and $B$ are disjoint and open and they cover $U\setminus\Gamma$. Moreover, $\partial A\cap U \subset\Gamma$ and $\partial B\cap U \subset\Gamma$. Indeed, if $x\in \partial A\cap U$, by continuity we have $u(x)\geq 0$. However, if $u(x)>0$, then by definition $x$ would belong to the interior of $A$. Therefore, necessarily $u(x)=0$, i.e.\ $x\in\Gamma$. The same argument applies if $x\in \partial B\cap U$. Hence, all the structural assumptions of Theorem~\ref{thm:glue-SVA} are satisfied for the pairs $(A,u)$ and $(B,v)$ on $V\Subset U$. Consequently, there exists a function $w\in W^{1,p}(V)\cap C(V)$ such that
\[
w=u \quad\text{on }A\cap V,\qquad 
w=0 \quad\text{on }(B\cup\Gamma)\cap V,\qquad 
\nabla w=\chi_{A}\,\nabla u \quad\text{a.e.\ in }V.
\]
Since $u\le 0$ on $B$ and $u=0$ on $\Gamma$, we deduce that $w=u^{+}$ in $V$. Moreover, since $u^{+}=u$ in $A_V:=A\cap V$ and $u^{+}=0$ in $V\setminus A_V$, it follows that $\nabla u^{+}=\chi_{A}\,\nabla u$ a.e.\ in $V$. Consequently,
\[
\|u^{+}\|_{L^{p}(V)}^{p}=\int_{A_V}|u|^{p},\qquad 
\|\nabla u^{+}\|_{L^{p}(V)}^{p}=\int_{A_V}|\nabla u|^{p},
\]
for all $1\le p<\infty$, while for $p=\infty$ one has
\[
\|u^{+}\|_{L^{\infty}(V)}\leq\|u\|_{L^{\infty}(A_V)},\qquad 
\|\nabla u^{+}\|_{L^{\infty}(V)}\le \|\nabla u\|_{L^{\infty}(A_V)}.
\]
These relations imply \eqref{eq:pos-Lp-generic}, with equalities when $1\le p<\infty$. This concludes the proof.

\end{proof}
\begin{remark}[Positive part on balls and half-balls]\label{regularity-positive-part-balls}
	\noindent Let $1\le p\le\infty$. Let $U\subset\mathbb{R}^n$ be open, $V\subset U$ open and $u\in C(U)$. Setting $A_U:=\{u>0\}\cap U$, we assume $u\in W^{1,p}(A_U)$.
	
	\noindent\textbf{(i) Ball.} If $U=B_1$ and $V=B_{1/2}$, 
    then
	\[
	u^+\in W^{1,p}(B_{1/2}),\qquad \nabla u^+=\chi_{\{u>0\}\cap B_1}\,\nabla u\ \text{ a.e. in }B_{1/2},
	\]
	and
	\[
	\|u^{+}\|_{L^{p}(B_{1/2})}\le \|u\|_{L^{p}(\{u>0\}\cap B_{1/2})},\qquad
	\|\nabla u^{+}\|_{L^{p}(B_{1/2})}\le \|\nabla u\|_{L^{p}(\{u>0\}\cap B_{1/2})},
	\]
	with equalities when $1\le p<\infty$. 
	
	\noindent\textbf{(ii) Half-ball.} If $U=B_1^{+}:=B_1\cap\{x_n>0\}$ and $V=B_{1/2}^{+}:=B_{1/2}\cap\{x_n>0\}$, 
    then
	\[
	u^+\in W^{1,p}(B_{1/2}^{+}),\qquad \nabla u^+=\chi_{\{u>0\}\cap B_1^{+}}\,\nabla u\ \text{ a.e. in }B_{1/2}^{+},
	\]
	and
	\[
	\|u^{+}\|_{L^{p}(B_{1/2}^{+})}\le \|u\|_{L^{p}(\{u>0\}\cap B_{1/2}^{+})},\qquad
	\|\nabla u^{+}\|_{L^{p}(B_{1/2}^{+})}\le \|\nabla u\|_{L^{p}(\{u>0\}\cap B_{1/2}^{+})},
	\]
	with equalities when $1\le p<\infty$. 
\end{remark}

\section{Lipschitz regularity of the one-phase Bernoulli problem}\label{s:10}

In this section, we present the proof of Theorem~\ref{thm:Lip-FBP}, which establishes the Lipschitz continuity of viscosity solutions to the one-phase Bernoulli free boundary problem~\eqref{eq:FBP1}. We begin by clarifying the precise meaning of a solution satisfying the free boundary condition in \eqref{eq:FBP1}.

\begin{definition}\label{def:grad-plus}
	Let $U\subset \mathbb{R}^{n}$ be an open set and define \(d(y):=\operatorname{dist}(y,\partial U)\).
	We say that \( \left| \nabla u^{+} \right| \le h \) on \(\mathcal{F}(u)\) if, for every \(y_{0}\in \mathcal{F}(u)\) and every
	\(0\le \varphi \in C\big(B_{\delta}(y_{0})\big)\) with \(0<\delta<d(y_{0})\), such that
	\[
	y_{0}\in \mathcal{F}(\varphi)\quad\text{and}\quad 
	\mathcal{F}(\varphi)\ \text{is of class } C^{1}\ \text{near } y_{0},
	\]
	and
	\[
	\varphi \in C^{1}\!\big(\{\varphi>0\}\cap \overline{B_{\delta}(y_{0})}\big)
	\]
	touches \(u^{+}\) from below at \(y_{0}\), then
	\[
	\frac{\partial \varphi}{\partial \nu}(y_{0}) \le h(y_{0}),
	\]
	where \(\nu\) is the unit normal vector at \(y_{0}\) pointing towards the interior of the set \(\{\varphi>0\}\).
\end{definition}
Having introduced the notion of solution to \eqref{eq:FBP1}, we are now ready to prove Theorem \ref{thm:Lip-FBP}.

\begin{proof}[Proof of Theorem~\ref{thm:Lip-FBP}]
	The argument is divided into five steps. We denote \(B_{1}^{+}(u):=B_{1}\cap\{u>0\}\).

    \smallskip
	\noindent\textbf{Step 1.}
	Let us assume \(0\in \mathcal{F}(u)=\partial\{u>0\}\cap B_{1}\). Fixed \(x_{0}\in B_{1/2}\cap\{u>0\}\), we set
	\begin{equation}\label{def-d0}
	d_{0}:=\inf\{|x_{0}-y|:\ y\in \mathcal{F}(u)\}.
	\end{equation}
	Since \(\mathcal{F}(u)\subset \overline{B_{1}}\) and \(0\in \mathcal{F}(u)\), it follows that 
\(
d_{0}\le |x_{0}|<\tfrac12
\).
By compactness of \(\mathcal{F}(u)\subset \overline{B_{1}}\), there exists \(y_{0}\in \mathcal{F}(u)\) such that 
\(
|x_{0}-y_{0}|=d_{0}=\operatorname{dist}(x_{0},\mathcal{F}(u))
\).
Consequently, \(B_{d_{0}}(x_{0})\subset \{u>0\}\). Indeed, if there existed \(p\in B_{d_{0}}(x_{0})\) with \(u(p)\le 0\), then by continuity along the segment \([x_{0},p]\), one could find \(z\in \mathcal{F}(u)\) with \(|x_{0}-z|<d_{0}\), which contradicts \eqref{def-d0}.

	\noindent\textbf{Step 2.}
	Let us define
	\[
	v(x):=\frac{u(x_{0}+d_{0}x)}{d_{0}},\qquad x\in B_{2}.
	\]
	Since \(B_{d_{0}}(x_{0})\subset\{u>0\}\), it follows that \(B_{1}\subset\{v>0\}\). The $v$ solves in the viscosity sense, the equation
	\[
	|\nabla v|^{\alpha}\,F_{d_{0}}(D^{2}v)=f_{d_{0}}\quad\text{in }B_{1},
    \]
    where 
    \[
	F_{d_{0}}(M):=d_{0}\,F\!\Big(\frac{M}{d_{0}}\Big)\qquad\text{and}\qquad 
	f_{d_{0}}(x):=d_{0}\,f(x_{0}+d_{0}x).
	\]
	Then \(F_{d_{0}}\) has the same ellipticity \((\lambda,\Lambda)\) as \(F\) and \(F_{d_{0}}(0)=0\). So, in particular \(\cM^-_{\lambda,\Lambda}\le F_{d_{0}}\le \cM^+_{\lambda,\Lambda}\). Therefore, by Theorem \ref{Thm:HarnackBD},
	\begin{equation}\label{eq-v-FB}
	\sup_{B_{1/2}} v\ \le\ C_{\mathrm H}\Big(\inf_{B_{1/2}} v+\|f_{d_{0}}\|_{L^{\infty}(B_{1})}^{\frac{1}{1+\alpha}}\Big),
	\end{equation}
	which implies
	\begin{equation}\label{useHar-FB}
	v(0)\ \le\ C_{\mathrm H}\Big(\inf_{B_{1/2}} v+d_{0}^{\frac{1}{1+\alpha}}\|f\|_{L^{\infty}(B_{1})}^{\frac{1}{1+\alpha}}\Big).
	\end{equation}
	
	\noindent\textbf{Step 3.}
	Let us fix
	\[
	A:=\|f_{d_{0}}\|_{L^{\infty}(B_{1})}^{\frac{1}{1+\alpha}}\qquad\text{and}\qquad
	y_{0}':=(y_{0}-x_{0})/d_{0}\in\partial B_{1}.
	\]
	So, \eqref{useHar-FB} reads as 
	\begin{equation}\label{useHar-FB2}
	v(0) \le\ C_{\mathrm H}\big(\inf_{B_{1/2}}v + A\big).
	\end{equation}
	 Let us fix also
	\[
	\theta:=\min\Big\{1,\Big(\frac{c_{0}}{2}\Big)^{\!\frac{1}{1+\alpha}}\Big\}\in(0,1]
    \]
    and we assume that
    \begin{equation}\label{eq:smallRHS-threshold}	C_{\mathrm H}A\ \le\ \frac{\theta}{2}\,v(0).
	\end{equation}
	Then by \eqref{useHar-FB2}, we have
	\begin{equation}\label{eq:M-def}
	\inf_{B_{1/2}}v\ \ge\ \frac{1}{C_{\mathrm H}}\,v(0)-A\ \ge\ \frac{1}{2C_{\mathrm H}}\,v(0)\ =:\ M\ >0,
	\end{equation}
	where \(M>0\) since \(x_{0}\in\{u>0\}\) and so \(v(0)=u(x_{0})/d_{0}>0\). Moreover, we have
	\begin{equation}
	\begin{aligned}\label{eq:strict-gap}
    2\|f_{d_0}\|_{L^\infty(B_1)}
		&=2A^{1+\alpha}
		\ \le\ 2\Big(\frac{\theta}{2C_{\mathrm H}}\,v(0)\Big)^{1+\alpha}
		\ =\ \frac{\theta^{1+\alpha}}{2^{\alpha}C_{\mathrm H}^{1+\alpha}}\,v(0)^{1+\alpha} \\[1mm]
		&\le\ \frac{c_0}{2^{1+\alpha}C_{\mathrm H}^{1+\alpha}}\,v(0)^{1+\alpha}
		\ =\ c_0\Big(\frac{v(0)}{2C_{\mathrm H}}\Big)^{1+\alpha}
		\ =\ c_0\,M^{1+\alpha}.
	\end{aligned}
\end{equation}
	
	Let \(\mathcal A_{1}:=B_{1}\setminus\overline{B_{1/2}}\). By Proposition~\ref{prop:barrier}, there exists \(\Gamma_{M}\in C^{\infty}(\overline{\mathcal A_{1}})\) such that
	\[
	\Gamma_{M}=0\ \text{on }\partial B_{1},\qquad
	\Gamma_{M}=M\ \text{on }\partial B_{1/2},\qquad
	|\nabla\Gamma_{M}|^{\alpha}\,\cM^-_{\lambda,\Lambda}(D^{2}\Gamma_{M})\ \ge\ c_{0}\,M^{1+\alpha}\ \text{ in }\mathcal A_{1},
	\]
	and \(|\nabla\Gamma_{M}|\ge C_{\mathrm{grad}}\,M\) on \(\overline{\mathcal A_{1}}\).
	Then on \(\partial\mathcal A_{1}\), by \eqref{eq:M-def}, we have \(v\ge M=\Gamma_{M}\) on \(\partial B_{1/2}\) and, since \(B_{1}\subset\{v>0\}\), it follows that \(v\ge 0=\Gamma_{M}\) on \(\partial B_{1}\). Hence, by the comparison principle with strict inequalities on right–hand sides (see, e.g.,\cite[Theorem 2.9]{Birindelli-Demengel2004}), we obtain \(v\ge\Gamma_{M}\) in \(\mathcal A_{1}\). For the sake of completeness, we present below a direct proof based on viscosity methods. Suppose \(\max_{\overline{\mathcal A_1}}(\Gamma_M - v) > 0\) and let \(x_{*}\in\mathcal A_1\) be an interior maximizer of \(w:=\Gamma_M-v\). Then \(v-\Gamma_M\) has a local minimum at \(x_{*}\) and, since \(v\) is a viscosity supersolution if \eqref{eq-v-FB} and \(\Gamma_M\in C^\infty\), we have
	\[
	|\nabla \Gamma_M(x_*)|^{\alpha}\,\cM^-_{\lambda,\Lambda}\big(D^{2}\Gamma_M(x_*)\big)\ \le\ \|f_{d_0}\|_{L^\infty(B_1)}.
	\]
	On the other hand, by the barrier property, we get
	\[
	|\nabla \Gamma_M(x_*)|^{\alpha}\,\cM^-_{\lambda,\Lambda}\big(D^{2}\Gamma_M(x_*)\big)\ \ge\ c_0\,M^{1+\alpha}\ >\ 0,
	\]
	which contradicts \eqref{eq:strict-gap}. Note that, since \eqref{eq:strict-gap} yields \(c_0 M^{1+\alpha} \ge 2\|f_{d_0}\|_{L^\infty(B_1)}\), the contradiction is strict whenever \(\|f_{d_0}\|_{L^\infty(B_1)}>0\). In the remaining case \(\|f_{d_0}\|_{L^\infty(B_1)}=0\), the result still follows from the fact that \(c_0 M^{1+\alpha}>0\). Therefore
	\begin{equation}\label{eq:v-ge-barrier}
	\Gamma_M \le v \qquad\text{in }\mathcal A_1.
	\end{equation}
    Now, let us consider the following extension of $\Gamma_M$
	\[
	\widetilde\Gamma_{M}(x):=
	\begin{cases}
		0, & x\notin B_{1},\\
		\Gamma_{M}(x), & x\in \mathcal A_{1},\\
		M, & x\in B_{1/2}.
	\end{cases}
	\]
	Let \(\delta>0\) be such that \(B_{\delta}(y_{0}')\) meets only the smooth arc of \(\partial B_{1}\). Then \(0\le\widetilde\Gamma_{M}\in C^{0}(B_{\delta}(y_{0}'))\), \(\mathcal{F}(\widetilde\Gamma_{M})\) is \(C^{1}\) near \(y_{0}'\), and \(\widetilde\Gamma_{M}\in C^{1}(\{\widetilde\Gamma_{M}>0\}\cap\overline{B_{\delta}(y_{0}')})\). By \eqref{eq:v-ge-barrier} and \(v\ge M\ge0\) in \(B_{1/2}\), we have \(\widetilde\Gamma_{M}\le v^{+}\) near \(y_{0}'\) and \(\widetilde\Gamma_{M}(y_{0}')=v^{+}(y_{0}')=0\). Thus \(\widetilde\Gamma_{M}\) touches \(v^{+}\) from below at \(y_{0}'\). So, by Definition~\ref{def:grad-plus}, if \(\nu\) is the inward normal to \(\{\widetilde\Gamma_{M}>0\}\) at \(y_{0}'\), we have
	\[
	\frac{\partial \widetilde\Gamma_{M}}{\partial \nu}(y_{0}')\ \le\ h_{d_{0}}(y_{0}')\ \le\ \|h\|_{L^{\infty}(\mathcal{F}(u))},\qquad
	\frac{\partial \widetilde\Gamma_{M}}{\partial \nu}(y_{0}')=|\nabla\Gamma_{M}(y_{0}')|\ge C_{\mathrm{grad}}\,M,
	\]
	which ensure that
	\begin{equation}\label{eq:M-le-h}
	M\ \le\ \frac{1}{C_{\mathrm{grad}}}\,\|h\|_{L^{\infty}(\mathcal{F}(u))}.
	\end{equation}
	
	\noindent\textbf{Step 4.}
	Now, if the condition in \eqref{eq:smallRHS-threshold} fails, then
	\[
	v(0)\ \le\ 2\theta^{-1}\,C_{\mathrm H}\,d_{0}^{\frac{1}{1+\alpha}}\,
	\|f\|_{L^{\infty}(B_{1})}^{\frac{1}{1+\alpha}}.
    \] By combing this with \eqref{eq:M-le-h} and the pointwise gradient estimate \eqref{eq:grad0} of Theorem \ref{thm:gradient-est} on \(v\) in \(B_{1}\),
    we obtain
    \begin{equation}\label{stima-gradv0}
	|\nabla v(0)| \le C_{0}\Big(\|h\|_{L^{\infty}(\mathcal{F}(u))}+d_{0}^{\frac{1}{1+\alpha}}\|f\|_{L^{\infty}(B_{1})}^{\frac{1}{1+\alpha}}\Big),
	\end{equation}
    for some suitable constant $C_0=C_0(n,\lambda,\Lambda,\alpha)$. Then, since \(\nabla v(0)=\nabla u(x_{0})\), by \eqref{stima-gradv0}, we conclude that for every \(x_{0}\in B_{1/2}\cap\{u>0\}\) and \(d_{0}=\operatorname{dist}(x_{0},\mathcal{F}(u))\),
	\begin{equation}\label{stimastep4}
	|\nabla u(x_{0})|
	\ \le\ C_{0}\Big(\|h\|_{L^{\infty}(\mathcal{F}(u))}+d_{0}^{\frac{1}{1+\alpha}}\|f\|_{L^{\infty}(B_{1})}^{\frac{1}{1+\alpha}}\Big).
	\end{equation}
	Moreover, since \(d_{0}\le 1\), it follows
	\[
	\sup_{B_{1/2}\cap\{u>0\}}|\nabla u|\ \le\ C_{0}\Big(\|h\|_{L^{\infty}(\mathcal{F}(u))}+\|f\|_{L^{\infty}(B_{1})}^{\frac{1}{1+\alpha}}\Big).
	\]
	Thus, in particular, by Theorem \ref{thm:gradient-est}, there exists \(\gamma=\gamma(n,\lambda,\Lambda,\alpha)\in(0,1]\), such that
	\[
	u\in C^{1,\gamma}\big(B_{1/2}\cap\{u>0\}\big),\quad\text{and}
\quad	\|\nabla u\|_{L^{\infty}(B_{1/2}\cap\{u>0\})}
	\ \le\ C_{0}\Big(\|h\|_{L^{\infty}(\mathcal{F}(u))}+\|f\|_{L^{\infty}(B_{1})}^{\frac{1}{1+\alpha}}\Big).
	\]

    \smallskip
	\noindent\textbf{Step 5.}
Fixed \(x_{0}\in B_{1/2}\cap\{u>0\}\), let us set \(d_{0}:=\operatorname{dist}(x_{0},\mathcal{F}(u))\), \(r_{*}:=\min\{d_{0},\tfrac12\}\) and 
\[v(x):=u(x_{0}+r_{*}x)/r_{*}\qquad  \text{for }\ x\in B_{1}.\]
Then \(B_{1}\subset\{v>0\}\) and $v$ solves in the viscosity sense
\[
|\nabla v|^{\alpha}\,F_{r_{*}}(D^{2}v)=f_{r_{*}}\quad\text{in }B_{1},
\]
where
\[
F_{r_{*}}(M):=r_{*}\,F\!\Big(\frac{M}{r_{*}}\Big)\qquad \text{and}\qquad
f_{r_{*}}(x):=r_{*}\,f(x_{0}+r_{*}x).
\]
Then \(F_{r_{*}}\) uniformly elliptic with the same \((\lambda,\Lambda)\) as \(F\) and \(F_{r_{*}}(0)=0\).
If \(r_{*}=d_{0}\le\tfrac12\), then by \eqref{stimastep4}, we have
\begin{equation}\label{stima-gra-1}
|\nabla u(x_{0})|\ \le\ C_1\Big(\|h\|_{L^{\infty}(\mathcal{F}(u))}+d_{0}^{\frac{1}{1+\alpha}}\|f\|_{L^{\infty}(B_{1})}^{\frac{1}{1+\alpha}}\Big),
\end{equation}
for some constant $C_1=C_1(n,\lambda,\Lambda,\alpha)>0$.
While, if \(r_{*}=\frac12\), so when \(d_{0}>\tfrac12\), then applying Theorem \ref{thm:gradient-est} to \(v\) on \(B_{1}\) and using \(\|v\|_{L^{\infty}(B_{1})}=\|u\|_{L^{\infty}(B_{r_{*}}(x_{0}))}/r_{*}\le 2\,\|u\|_{L^{\infty}(B_{1}^{+}(u))}\) and \(\|f_{r_{*}}\|_{L^{\infty}(B_{1})}\le \tfrac12\|f\|_{L^{\infty}(B_{1})}\), we obtain
\begin{equation}\label{stima-grad-2}
|\nabla u(x_{0})|\ \le\ C\Big(\|u\|_{L^{\infty}(B_{1}^{+}(u))}+\|f\|_{L^{\infty}(B_{1})}^{\frac{1}{1+\alpha}}\Big),
\end{equation}
where $C=C(n,\lambda,\Lambda,\alpha)>0$. So combining \eqref{stima-gra-1} and \eqref{stima-grad-2}, there exists \(C_{0}=C_{0}(n,\lambda,\Lambda,\alpha)>0\) such that, for every \(x_{0}\in B_{1/2}\cap\{u>0\}\),
\[
|\nabla u(x_{0})|\ \le\ C_{0}\Big(\|h\|_{L^{\infty}(\mathcal{F}(u))}+\|u\|_{L^{\infty}(B_{1}^{+}(u))}+\|f\|_{L^{\infty}(B_{1})}^{\frac{1}{1+\alpha}}\Big).
\]
Thus, in particular, by Theorem \ref{thm:gradient-est}, we have \(u\in C^{1,\gamma}(B_{1/2}\cap\{u>0\})\) and
\[
\|\nabla u\|_{L^{\infty}(B_{1/2}\cap\{u>0\})}\ \le\ C_{0}\Big(\|h\|_{L^{\infty}(\mathcal{F}(u))}+\|u\|_{L^{\infty}(B_{1}^{+}(u))}+\|f\|_{L^{\infty}(B_{1})}^{\frac{1}{1+\alpha}}\Big).
\]
Finally, in light of  Remark~\ref{regularity-positive-part-balls}, since \(u\in C(B_{1})\) and \(u\in W^{1,\infty}(\{u>0\}\cap B_{1})\) with the bound above, we conclude
\[
u^{+}\in C^{0,1}(B_{1/2}),\qquad
\|\nabla u^{+}\|_{L^{\infty}(B_{1/2})}\ \le\ C_{0}\Big(\|h\|_{L^{\infty}(\mathcal{F}(u))}+\|u\|_{L^{\infty}(B_{1}^{+}(u))}+\|f\|_{L^{\infty}(B_{1})}^{\frac{1}{1+\alpha}}\Big).
\]
When \(0\in \mathcal{F}(u)\), the argument employed in Steps 3 and 4 shows that the term \(\|u\|_{L^{\infty}(B_{1}^{+}(u))}\) can be omitted.
\end{proof}

\section{Uniform Lipschitz continuity for a flame propagation model} \label{s:11}

In this section, we present the proof of Theorem \ref{thm:flame-deg}, which establishes the Lipschitz regularity of solutions to the flame propagation problem \eqref{eq:flame_prop}.
\begin{proof}
	\noindent By Theorem \ref{thm:gradient-est}, any bounded viscosity solution of
	\(
	|\nabla u|^{\alpha}F(D^2u)=g
	\)
	with $g\in C\cap L^\infty$ belongs to $C^{1,\gamma}_{\mathrm{loc}}(B_1)$
	for some $\gamma \in (0,1)$ depending on $n, \lambda, \Lambda$ and $\alpha$.
	\noindent Here \(g(x):=\beta_\varepsilon(u_\varepsilon(x))+f(x)\) is continuous and bounded, so \(u_\varepsilon\in C^{1,\gamma}_{\mathrm{loc}}(B_1)\).
	
	\medskip
	\noindent Fix $\varepsilon\in(0,\tfrac18]$ and $x_0\in B_{1/2}$. We argue by cases.
	
	\medskip
	\noindent\textbf{Case A: $x_0\in\{0\le u_\varepsilon\le\varepsilon\}$.}
	\noindent Since $B_\varepsilon(x_0)\subset B_1$, define
	\[
	v(x):=\frac{1}{\varepsilon}\,u_\varepsilon(x^{}_0+\varepsilon x),\qquad x\in B_1.
	\]
	By Remark \ref{rem:rescaling}, with $a=\tfrac{1}{\varepsilon}$ and $b=\varepsilon$, we get
	\begin{equation}\label{eq:RS-CaseA}
		\nabla v(x)=\nabla u_\varepsilon(x_0+\varepsilon x),\qquad
		|\nabla v|^{\alpha}\,\widetilde F(D^2v)=\beta(v)+\varepsilon f(x_0+\varepsilon x)\quad\text{in }B_1,
	\end{equation}
	\noindent where $\widetilde F(M)=\varepsilon\,F(M/\varepsilon)$ is $(\lambda,\Lambda)$-elliptic. Since $v(0)=u_\varepsilon(x_0)/\varepsilon\le1$ and $\|\beta(v)\|_{L^\infty(B_1)}\le\|\beta\|_{L^\infty(\R)}$, 
	by the pointwise gradient estimate \eqref{eq:grad0} there exists a constant \(C_A=C_A(n,\lambda,\Lambda,\alpha)\), independent of $\varepsilon$ and $\beta$, such that
	\begin{equation}\label{eq:grad-CaseA}
	|\nabla u_\varepsilon(x_0)|=|\nabla v(0)|
	\le C_A\Big(1+\|\beta\|_{L^\infty(\R)}^{\frac{1}{1+\alpha}}
	+\|f\|_{L^\infty(B_1)}^{\frac{1}{1+\alpha}}\Big).
	\end{equation}
	
	\medskip
	\noindent\textbf{Case B:} Assume $x_0\in B_{1/2}$ with $u_\varepsilon(x_0)>\varepsilon$ and $u_\varepsilon(0)\le\varepsilon$. Setting
	\[
	E_\varepsilon:=\{u_\varepsilon\le\varepsilon\}\cap\overline{B_{1/2}},
	\]
	we choose $y_0\in E_\varepsilon$ such that
	\begin{equation}\label{eq:d0}
		d_0:=|x_0-y_0|=\operatorname{dist}(x_0,E_\varepsilon)>0.
	\end{equation}
    Consequently, \(B_{d_{0}}(x_{0})\subset \{u_\varepsilon>\varepsilon\}\). Indeed, if there existed \(p\in B_{d_{0}}(x_{0})\) with \(u_\varepsilon(p)\le \varepsilon\), then by continuity along the segment \([x_{0},p]\), one could find \(z\in \mathcal{F}(u)\) with \(|x_{0}-z|<d_{0}\), which contradicts \eqref{eq:d0}. Hence
	\begin{equation}\label{eq:ball-in-pos}
		|\nabla u_\varepsilon|^\alpha F(D^2u_\varepsilon)=f \ \text{ in } B_{d_0}(x_0).
	\end{equation}
	Let us define $w:=u_\varepsilon-\varepsilon\ge0$ in $B_{d_0}(x_0)$ and the rescaled function
	\[
	v(x):=\frac{1}{d_0}\,w(x_0+d_0 x),\qquad \text{for }\ x\in B_1.
	\]
	A direct computation yields
	\begin{equation}\label{eq:RS-id}
		\nabla v(x)=\nabla u_\varepsilon(x_0+d_0x)\qquad\text{and}\qquad
		D^2v(x)=d_0\,D^2u_\varepsilon(x_0+d_0x).
	\end{equation}
	By \eqref{eq:ball-in-pos}, the function $v$ solves in the viscosity sense
	\begin{equation}\label{eq:RS-PDE}
		|\nabla v|^\alpha\,\widetilde F(D^2v)=d_0\,f(x_0+d_0x)\qquad\text{in }B_1,	\end{equation}
	where $\widetilde F(M):=d_0\,F\!\left(\frac{1}{d_0}M\right)$ is $(\lambda,\Lambda)$-elliptic. Note also that $v\ge0$ in $B_1$ and
	\begin{equation}\label{eq:touch-point}
		v(x_*)=0 \quad\text{for}\quad x_*:=\frac{y_0-x_0}{d_0}\in\partial B_1 .
	\end{equation}
	By the boundary gradient estimate \eqref{eq:grad0-bdry} to $v$ (with datum $d_0\,f(x_0+d_0\cdot)$), since $d_0\le \frac12$, there exists a constant 
	\(
	C_1=C_1(n,\lambda,\Lambda,\alpha)
	\)
	independent of $\varepsilon$ and $\beta$ such that
	\begin{equation}\label{eq:bdry-grad}
		|\nabla u_\varepsilon(x_0)|=|\nabla v(0)|
		\le C_1\Big(\partial_\nu v(x_*)+\|f\|_{L^\infty(B_1)}^{\frac{1}{1+\alpha}}\Big).
	\end{equation}
	Moreover, by \eqref{eq:RS-id} and $\partial_\nu v(x_*)\le|\nabla v(x_*)|$, we have
	\begin{equation}\label{eq:normals-to-grad}
		\partial_\nu v(x_*)\le |\nabla v(x_*)|=|\nabla u_\varepsilon(y_0)|.
	\end{equation}
	
	On the other hand, since $y_0\in E_\varepsilon\subset\{0\le u_\varepsilon\le\varepsilon\}\cap B_{1/2}$, by \eqref{eq:grad-CaseA}
	\begin{equation}\label{eq:CaseA-grad}
		|\nabla u_\varepsilon(y_0)|
		\le C_A\Big(1+\|\beta\|_{L^\infty(\R)}^{\frac{1}{1+\alpha}}
		+\|f\|_{L^\infty(B_1)}^{\frac{1}{1+\alpha}}\Big),
	\end{equation}
    with $C_A=C_A(n,\lambda,\Lambda,\alpha)$ independent of $\varepsilon$ and $\beta$. 
    
	So, combining \eqref{eq:bdry-grad}, \eqref{eq:normals-to-grad} and \eqref{eq:CaseA-grad}, we obtain
	\begin{equation}\label{eq:grad-CaseB}
		|\nabla u_\varepsilon(x_0)|
		\le C_B\Big(1+\|\beta\|_{L^\infty(\R)}^{\frac{1}{1+\alpha}}
		+\|f\|_{L^\infty(B_1)}^{\frac{1}{1+\alpha}}\Big),
	\end{equation}
	for some suitable constant \(C_B=C_B(n,\lambda,\Lambda,\alpha)\), independent of $\varepsilon$ and $\beta$.
	
	\medskip
	\noindent\textbf{Unified local estimate.}
	\noindent Let
	\[
	E_\varepsilon:=\{u_\varepsilon\le\varepsilon\}\cap\overline{B_{1/2}},\qquad\text{and}\qquad 
	d_0:=\dist(x_0,E_\varepsilon).
	\]
	\noindent We treat three regimes:
	
	\smallskip
	\noindent\emph{(1) Transition layer: $0\le u_\varepsilon(x_0)\le\varepsilon$.}
	\noindent Defining $v(x)=u_\varepsilon(x_0+\varepsilon x)/\varepsilon$ and applying \eqref{eq:grad0} as in \eqref{eq:grad-CaseA}, we conclude that
	\begin{equation}\label{eq:unif-case1}
		|\nabla u_\varepsilon(x_0)|
		\le C_{\mathrm{tr}}\Big(1+\|\beta\|_{L^\infty(\R)}^{\frac{1}{1+\alpha}}
		+\|f\|_{L^\infty(B_1)}^{\frac{1}{1+\alpha}}\Big),
	\end{equation}
	for some constant $C_{\mathrm{tr}}=C_{\mathrm{tr}}(n,\lambda,\Lambda,\alpha)$,
	independent of $\varepsilon$ and $\beta$.
	
	\smallskip
	\noindent\emph{(2) Near the $\varepsilon$–level and $u_\varepsilon(x_0)>\varepsilon$.} Let us assume that  $d_0\le\tfrac{1}{10}$ and let $y_0\in E_\varepsilon$ attaining $d_0$. Following the same argument employed in \textbf{Case B} together with \eqref{eq:unif-case1} at $y_0$, we get
	\begin{equation}\label{eq:unif-case2}
		|\nabla u_\varepsilon(x_0)|
		\le C_{\mathrm{near}}\Big(1+\|\beta\|_{L^\infty(\R)}^{\frac{1}{1+\alpha}}
		+\|f\|_{L^\infty(B_1)}^{\frac{1}{1+\alpha}}\Big),
	\end{equation}
	where $C_{\mathrm{near}}=C_{\mathrm{near}}(n,\lambda,\Lambda,\alpha)$, is independent of $\varepsilon$ and $\beta$.
	
	\smallskip
	\noindent\emph{(3) Interior regime.} Suppose that $d_0>\frac{1}{10}=:r$. Setting $\tilde v(x)=u_\varepsilon(x_0+rx)/r$ and apply \eqref{eq:grad0} to $\tilde v$, we obtain
	\begin{equation}\label{eq:unif-case3}
		|\nabla u_\varepsilon(x_0)|
		\le C_{\mathrm{far}}\Big(\tfrac{1}{r}\,\|u_\varepsilon\|_{L^\infty(B_1)}
		+\|f\|_{L^\infty(B_1)}^{\frac{1}{1+\alpha}}\Big),
	\end{equation}
	\noindent with $C_{\mathrm{far}}=C_{\mathrm{far}}(n,\lambda,\Lambda,\alpha)$,
	independent of $\varepsilon$ and of $\beta$.
	
	\smallskip
	In conclusion, by \eqref{eq:unif-case1}, \eqref{eq:unif-case2} and \eqref{eq:unif-case3}, we deduce
	\[
	\|\nabla u_\varepsilon\|_{L^\infty(B_{1/2})}
	\le C_*\Big(
	1+\|\beta\|_{L^\infty(\R)}^{\frac{1}{1+\alpha}}
	+\|f\|_{L^\infty(B_1)}^{\frac{1}{1+\alpha}}
	+\|u_\varepsilon\|_{L^\infty(B_1)}
	\Big),
	\]
	\noindent where
	\[
	C_*=\max\{C_{\mathrm{tr}},C_{\mathrm{near}},C_{\mathrm{far}},C_A,C_B\}
	=C_*(n,\lambda,\Lambda,\alpha)
	\]
	\noindent is independent of $\varepsilon$ and $\beta$.
\end{proof}

\appendix
\section{From Campanato to \(C^{1,\gamma}\) Regularity} \label{appendixA}
In this Appendix, we present some estimates relating different notions of $C^{1,\gamma}$.
 This appendix provides a self-contained discussion that reconciles the notions of classical \(C^{1,\gamma}\) regularity, \(C^{1,\gamma}\) via polynomial approximation and Campanato \(C^{1,\gamma}\) in the sense of \cite[Appendix 1]{campanato}. Although the results are well known, especially in the $C^{0,\gamma}$ framework, we include here explicit proofs and quantitative estimates, which are used throughout the paper.

\begin{definition}[Modulus of continuity and $C^{0,\omega}$ seminorm]
	A \emph{modulus of continuity} is a function $\omega:[0,\infty)\to[0,\infty)$ that is
	continuous, nondecreasing, subadditive, and satisfies $\omega(0)=0$.
	We say that $\omega$ is \emph{strictly positive} if $\omega(t)>0$ for every $t>0$.
	
	For a set $E\subset\mathbb R^n$ and $f:E\to\mathbb R$, define the $C^{0,\omega}$-seminorm as
	\begin{equation}\label{defCamp}
	[f]_{C^{0,\omega}(E)}
	:=\inf\Big\{\,C\ge0:\ |f(x)-f(y)|\le C\,\omega(|x-y|)\ \ \forall\,x,y\in E\,\Big\}.
	\end{equation}
    \end{definition}
    \begin{remark}
	If $\omega$ is strictly positive, then the infimum in \eqref{defCamp}, coincides with the supremum quotient
	\[
	[f]_{C^{0,\omega}(E)}
	=\sup_{\substack{x,y\in E\\ x\ne y}}\frac{|f(x)-f(y)|}{\omega(|x-y|)}.
	\]
\end{remark}

We begin by establishing a quantitative criterion that connects pointwise $C^{1,\omega}$ affine expansions with the classical $C^{1,\omega}$ regularity of a function, providing explicit bounds for the gradient and its modulus of continuity.

\begin{theorem}[Uniform pointwise $C^{1,\omega}$ estimates]\label{thm:UTE-C1omega}
	Let $0<\rho<R$, $0<\sigma\le R-\rho$ and $T>0$. Let $\omega$ be a modulus of continuity and let $u\in L^\infty(B_R)$. Assume that for every $x_0\in B_\sigma$ there exists an affine function
	\[
	\ell_{x_0}(x):=u(x_0)+p(x_0)\cdot(x-x_0)
	\]
	such that
	\begin{equation}\label{eq:H}
		|u(x)-\ell_{x_0}(x)|\le T\,|x-x_0|\,\omega(|x-x_0|)\qquad \text{for all  }\ x\in B_\rho(x_0).
	\end{equation}
	Then $u\in C^{1,\omega}(B_\sigma)$ and
	\begin{equation}\label{eq:Linf}
		\|\nabla u\|_{L^\infty(B_\sigma)}\le \frac{2\|u\|_{L^\infty(B_R)}}{\rho}+T\,\omega(\rho/2).
	\end{equation}
	Moreover, if $B_{\rho/2}\subset B_\sigma$, then
	\begin{equation}\label{eq:holder-semi}
		[\nabla u]_{C^{0,\omega}(B_{\rho/2})}\le 8\,T.
	\end{equation}
	If, in addition,
	\begin{equation}\label{eq:Sbound}
		\sup_{x\in B_\sigma}|\nabla u(x)|=\sup_{x\in B_\sigma}|p(x)|\le S,
	\end{equation}
	and $\omega$ is strictly positive, then
	\begin{equation}\label{eq:global-S}
		[\nabla u]_{C^{0,\omega}(B_\sigma)}\le 8\,T+\frac{2S}{\omega(\rho)}.
	\end{equation}
	Without assuming \eqref{eq:Sbound} one still has
	\begin{equation}\label{eq:global-noS}
		[\nabla u]_{C^{0,\omega}(B_\sigma)} \le\ 10\,T+\frac{4\,\|u\|_{L^\infty(B_R)}}{\rho\,\omega(\rho)}.
	\end{equation}
	All constants above are explicit and independent of $\sigma$.
\end{theorem}

\begin{proof}
	First of all, we observe that \eqref{eq:H} implies differentiability at every $x_0\in B_\sigma$. Indeed, we have
	\[
	\frac{|u(x)-u(x_0)-p(x_0)\cdot(x-x_0)|}{|x-x_0|}\le T\,\omega(|x-x_0|)\to0, \quad\text{as }\ x\to x_0,
	\]
	and, in particular, $\nabla u(x_0)=p(x_0)$.
    
    To prove \eqref{eq:Linf}, fixed $x_0\in B_\sigma$, $e\in\mathbb S^{n-1}$, we set $s:=\rho/2$, $x_\pm:=x_0\pm se$.
	From \eqref{eq:H} at the center $x_0$, we write
	\[
	u(x_\pm)=\ell_{x_0}(x_\pm)+E_\pm\qquad\text{with}\qquad |E_\pm|\le T\,s\,\omega(s).
	\]
	Hence, since $\ell_{x_0}$ is affine, we get
	\[
	2s\,p(x_0)\!\cdot\!e=\ell_{x_0}(x_+)-\ell_{x_0}(x_-)=\big(u(x_+)-u(x_-)\big)-\big(E_+-E_-\big),
	\]
	which implies
	\[
	|p(x_0)\!\cdot\!e|\le \frac{|u(x_+)-u(x_-)|}{2s}+T\,\omega(s)\le \frac{2\|u\|_{L^\infty(B_R)}}{\rho}+T\,\omega(\rho/2).
	\]
	Then, taking the supremum over all directions $e\in\mathbb{S}^{n-1}$, we get \eqref{eq:Linf}.
	
	For the Hölder control, let $x,y\in B_\sigma$ with $d:=|x-y|\le \rho$ and set $z_0:=(x+y)/2$. Let us consider the affine difference
\[
H := \ell_x - \ell_y.
\]
Then $\nabla H = \nabla u(x) - \nabla u(y)$. For any $z \in B_{d/2}(z_0)$ we have $|z - x| \le d$ and $|z - y| \le d$, and recalling \eqref{eq:H}, we obtain
\[
|H(z)| \le |u(z) - \ell_x(z)| + |u(z) - \ell_y(z)| \le 2T\,d\,\omega(d).
\]
	Furthermore, since $H$ is affine, for every $r>0$, we have
	\begin{equation}\label{eq:affine}
		|\nabla H|=\frac{1}{r}\sup_{w\in\partial B_r(z_0)}|H(w)-H(z_0)|\le \frac{2}{r}\,\|H\|_{L^\infty(B_r(z_0))}.
	\end{equation}
	So, choosing $r=d/2$ yields
	\begin{equation}\label{est-loc-grad}
	|\nabla u(x)-\nabla u(y)|=|\nabla H|\le \frac{2}{d/2}\,\big(2T\,d\,\omega(d)\big)=8\,T\,\omega(d).
	\end{equation}
	If $B_{\rho/2}\subset B_\sigma$, taking the supremum over $x\neq y$ in $B_{\rho/2}$, we obtain \eqref{eq:holder-semi}.

	Finally, let us prove the global estimate \eqref{eq:global-S} under the assumption \eqref{eq:Sbound}. Let $x,y\in B_\sigma$. If $|x-y|\le \rho$, by \eqref{est-loc-grad}, we have
	\[
	\frac{|\nabla u(x)-\nabla u(y)|}{\omega(|x-y|)}\ \le\ 8T.
	\]
	If $|x-y|>\rho$, then $|\nabla u(x)-\nabla u(y)|\le 2S$ and $\omega(|x-y|)\ge \omega(\rho)$, hence
	\begin{equation}\label{use-S-1}
	\frac{|\nabla u(x)-\nabla u(y)|}{\omega(|x-y|)}\ \le\ \frac{2S}{\omega(\rho)}.
	\end{equation}
	Therefore,
	\begin{equation}\label{use-S-2}
	[\nabla u]_{C^{0,\omega}(B_\sigma)}
	= \sup_{x\neq y}\frac{|\nabla u(x)-\nabla u(y)|}{\omega(|x-y|)}
	 \le \max\left\{\,8T,\ \frac{2S}{\omega(\rho)}\right\}
	 \le 8T+\frac{2S}{\omega(\rho)},
	\end{equation}
	which is \eqref{eq:global-S}. Without the assumption \eqref{eq:Sbound}, by replacing $S$ in \eqref{use-S-1} and in \eqref{use-S-2} with the bound from \eqref{eq:Linf}, we obtain \eqref{eq:global-noS}.
\end{proof}

\begin{remark}[Power-type modulus of continuity $\omega(t)=t^\gamma$]\label{power-modulus}
	Assume $\omega(t)=t^\gamma$ with some $\gamma\in (0,1]$. Then \eqref{eq:H} is a uniform pointwise $C^{1,\gamma}$ expansion and the theorem yields $u\in C^{1,\gamma}(B_\sigma)$ with explicit bounds:
	\[
	\|\nabla u\|_{L^\infty(B_\sigma)}\ \le\ \frac{2\|u\|_{L^\infty(B_R)}}{\rho}\ +\ T\,\Big(\tfrac{\rho}{2}\Big)^{\!\gamma},
	\qquad
	[\nabla u]_{C^{0,\gamma}(B_{\rho/2})}\ \le\ 8\,T\ \quad \text{if }B_{\rho/2}\subset B_\sigma.
	\]
	If $\sup_{B_\sigma}|\nabla u|\le S$, then by \eqref{eq:global-S}, we also have
	\[
	[\nabla u]_{C^{0,\gamma}(B_\sigma)}\ \le\ 8\,T\ +\ \frac{2S}{\rho^{\gamma}}
	\ \le\ 8\Big(1+\rho^{-\gamma}\Big)\big(T+S\big),
	\]
	and, in general by \eqref{eq:global-noS}, it follows
	\[
	[\nabla u]_{C^{0,\gamma}(B_\sigma)}\ \le\ 10\,T\ +\ \frac{4\,\|u\|_{L^\infty(B_R)}}{\rho^{1+\gamma}}
	\ \le\ 10\Big(1+\rho^{-(1+\gamma)}\Big)\big(T+\|u\|_{L^\infty(B_R)}\big).
	\]
	Note that all constants are independent of $\sigma$.
\end{remark}

The previous result provides a quantitative passage from pointwise \(C^{1,\omega}\) expansions to classical \(C^{1,\omega}\) regularity. 
We now show that, conversely, a uniform Campanato control yields such pointwise expansions, thus connecting the two notions of \(C^{1,\gamma}\) regularity.

\begin{proposition}\label{prop:campanato-uniform-taylor}
	Fix $\gamma\in(0,1]$. Suppose $u:B_1\to\mathbb R$ satisfies
	\begin{equation}\label{eq:campanato}
	[u]_{1+\gamma,B_1}
	:=\sup_{\rho>0,\ x\in B_1}\ \inf_{\ell\ \mathrm{affine}}
	\sup_{z\in B_\rho(x)} \rho^{-1-\gamma}\,|u(z)-\ell(z)|
	\ \le A.
	\end{equation}	Then for every $x_0\in B_{1/2}$ there exists an affine function
	\[
	\ell_{x_0}(x)=u(x_0)+p(x_0)\cdot(x-x_0)
	\]
	such that
	\[
	|u(x)-\ell_{x_0}(x)|\ \le\ C(n,\gamma)\,A\,|x-x_0|^{1+\gamma}
	\qquad\text{for all }x\in B_{1/4}(x_0).
	\]
\end{proposition}

\begin{proof}
	By the definition of the Campanato seminorm (and since the infimum need not be attained), for every $x\in B_1$ and $r>0$ with $B_r(x)\subset B_1$, there exists an affine function $\ell$ such that
	\[
	\sup_{B_r(x)}|u-\ell|\ \le\ 2A\,r^{1+\gamma}.
	\]
    Moreover, \eqref{eq:campanato} readily implies the continuity of $u$.
	Fix $x_0\in B_{1/2}$ and set $r_k:=2^{-k}$. For each $k\geq1 $ we have $B_{r_k}(x_0)\subset B_1$. 
    Choose 
	\begin{equation} \label{eq:u-lk}
	    	\ell_k(x)=p_k\cdot(x-x_0)+c_k
	\qquad\text{with}\qquad
	\sup_{B_{r_k}(x_0)}|u-\ell_k|\ \le\ 2A\,r_k^{1+\gamma}.
	\end{equation}
	Since $B_{r_{k+1}}(x_0)\subset B_{r_k}(x_0)$, we have
	\begin{equation} \label{eq:suplk}
	    \sup_{B_{r_{k+1}}(x_0)}|\ell_k-\ell_{k+1}|
	\ \le\ \sup|u-\ell_k|+\sup|u-\ell_{k+1}|
	\ \le\ 4A\,r_k^{1+\gamma}.
	\end{equation}
	where $\ell_k-\ell_{k+1}=(p_k-p_{k+1})\cdot(x-x_0)+(c_k-c_{k+1})$.  
	Since $u$ is continuous and $B_{r_{k+1}}(x_0)\Subset B_{r_k}(x_0)\subset B_1$, the supremum in \eqref{eq:suplk} can be taken over the closure $\overline{B_{r_{k+1}}(x_0)}$.  
	Hence, we may evaluate \eqref{eq:suplk} at the boundary points
	\[
	x_\pm:=x_0\pm r_{k+1} e \in\partial B_{r_{k+1}}(x_0)\subset B_{r_k}(x_0)
	\qquad \text{where }\
	e:=\frac{p_k-p_{k+1}}{|p_k-p_{k+1}|}\in\mathbb S^{n-1},
	\]
	so that $x_\pm-x_0=\pm r_{k+1}e$. Evaluating \eqref{eq:suplk} at $x_+$ and $x_-$ gives
	\[
	\big|\,(p_k-p_{k+1})\cdot(r_{k+1}e)+(c_k-c_{k+1})\,\big|\ \le\ 4A\,r_k^{1+\gamma},
	\]
	\[
	\big|\,-(p_k-p_{k+1})\cdot(r_{k+1}e)+(c_k-c_{k+1})\,\big|\ \le\ 4A\,r_k^{1+\gamma}.
	\]
	Adding these two inequalities and using the triangle inequality on the left-hand side, we get
	\[
	2\,r_{k+1}\,|p_k-p_{k+1}|\ \le\ 8A\,r_k^{1+\gamma},
	\]
	hence, since $r_{k+1}=r_k/2$,
	\begin{equation} \label{eq:pk}
	|p_k-p_{k+1}|\ \le\ 8A\,r_k^{\gamma}.
	\end{equation}
	Evaluating \eqref{eq:suplk} instead at $x=x_0$ gives directly
	\begin{equation} \label{eq:ck}
	|c_k-c_{k+1}| = |\ell_k(x_0)-\ell_{k+1}(x_0)| \le 4A\,r_k^{1+\gamma}.
	\end{equation}
	For integers $m<n$, telescoping \eqref{eq:pk} and \eqref{eq:ck} and using the geometric sums of the dyadic radii, we obtain
	\begin{equation} \label{eq:estimpkck}
	    \begin{aligned}
	        	|p_m-p_n|&\le \sum_{j=m}^{n-1}|p_j-p_{j+1}|\le 8A\sum_{j=m}^{\infty} r_j^{\gamma}\le C(\gamma)\,A\,r_m^{\gamma}\xrightarrow[m\to\infty]{}0, \\ 
                	|c_m-c_n|&\le \sum_{j=m}^{n-1}|c_j-c_{j+1}|\le 4A\sum_{j=m}^{\infty} r_j^{1+\gamma}\le C(\gamma)\,A\,r_m^{1+\gamma}\xrightarrow[m\to\infty]{}0.
	    \end{aligned}
	\end{equation}
	Thus $(p_k)_{k \in \N}$ and $(c_k)_{k \in \N }$ are Cauchy sequences, therefore convergent. Write $p(x_0):=\lim_{k\to \infty} p_k$ and $c_\infty(x_0):=\lim_{k\to \infty} c_k$.  
	From \eqref{eq:u-lk} evaluated at $x_0$, we also have $|u(x_0)-c_k|\le 2A r_k^{1+\gamma}$, hence $c_\infty(x_0)=u(x_0)$.
	
	Let us define $\ell_{x_0}(x):=u(x_0)+p(x_0)\cdot(x-x_0)$.  
	For any $x\in B_{1/4}(x_0)$ choose $m$ with $r_{m+1}<|x-x_0|\le r_m$.  
 This leads to 
    \begin{align*}
        &|u(x)-\ell_{x_0}(x)| \leq |u(x)-\ell_{m}(x)| + |\ell_{m}(x)-\ell_{x_0}(x)| \\
        &\leq 2Ar_{m}^{1+\gamma} + |p_m-p(x_0)| |x-x_0| + |c_m-u(x_0)| \leq 2Ar_{m}^{1+\gamma} + C(\gamma)Ar_{m}^{1+\gamma} \leq  C(\gamma)\,A\,r_m^{1+\gamma}.
    \end{align*}
    where the first term is estimated by \eqref{eq:u-lk} and the second with \eqref{eq:estimpkck}.
    
	Since $r_m\sim |x-x_0|$, we conclude
	\[
	|u(x)-\ell_{x_0}(x)|\le C(n,\gamma)\,A\,|x-x_0|^{1+\gamma},
	\]
	which is the desired first--order Taylor expansion with uniform remainder in $B_{1/4}(x_0)$.
\end{proof}

\section*{Acknowledgements}
 The authors would like to thank David Jesus for fruitful discussions on the topic of this paper.

D.G. and E.M.M. are member of the {\em Gruppo Nazionale per l'Analisi Ma\-te\-ma\-ti\-ca, la Probabilit\`a e le loro Applicazioni} (GNAMPA) of the {\em Istituto Nazionale di Alta Matematica} (INdAM). D.G. is supported by  PRIN 2022 7HX33Z - CUP J53D23003610006, by University of Bologna funds for the project “Attività di collaborazione con università del Nord America” within the framework of “Interplaying problems in analysis and geometry” and by the INDAM-GNAMPA project 2025 “At The Edge Of Ellipticity” - CUP E5324001950001. E.M.M. is partially supported by INDAM-GNAMPA project 2025: \say{Ottimizzazione spettrale, geometrica e funzionale} - CUP: E5324001950001 and by the PRIN project 2022R537CS \say{NO$^3$--Nodal Optimization, NOnlinear elliptic equations, NOnlocal geometric problems, with a focus on regularity} - CUP J53D23003850006, founded by the European Union - Next Generation EU. 


\subsection*{Conflict of interest}
The authors confirm that they do not have actual or potential conflict of interest in
relation to this publication.

\subsection*{Author contributions}

All authors shared equally the writing and the reviewing of the manuscript.

\subsection*{Data availability statement} Not applicable.

\subsection*{Ethical Approval} Not applicable.

\bibliographystyle{plain} 
\bibliography{bib} 
\Addresses
\end{document}